\documentclass[1p,a4paper]{elsarticle_modified}

\usepackage{amsfonts}
\usepackage{amssymb}
\usepackage{amsmath}
\usepackage{amsthm}
\usepackage{amsbsy}
\usepackage{bm}
\usepackage{caption,subcaption}
\usepackage[colorlinks]{hyperref}
\usepackage{graphicx}
\usepackage{xcolor} %% white, black, red, green, blue, cyan, magenta, yellow
\usepackage{float}
\usepackage{tikz}
\usepackage{microtype} %resolve bad \hbox full too wide error 

\graphicspath{{figures/}} % For \input command

\usepackage[esvecte]{overarrows}

\NewOverArrowCommand{vect}{
	start={\relbaredd},
	middle={\relbareda}, end={\fldr},
	space after arrow=-.15ex,
}

\NewOverArrowCommand{edge}{
	start={\relbaredd},
	middle={\relbareda}, end={\relbaredd},
		trim end=0,
	space after arrow=-.3ex
}

\def\H{{\mathbb{H}^3}}

\def\.{\mskip1mu} %spacing

\def\edgepath#1{\edge{#1}}
\def\polygon#1{(#1)}

\def\V{{\bm{V}}}

\def\N_#1{{N^\boxtimes_{#1}}}

\def\tN_#1{{N^\squareslash_{#1}}}

\def\bdN_#1{{N^\square_{#1}}}

\def\ssetminus{{\,\setminus\,}}

\newlength\squareheight
\newcommand\squareslash{{\tikz{\draw (0,0) rectangle (\squareheight,\squareheight);\draw(0,0) -- (\squareheight,\squareheight)}}}

\def\RR{\mathbb{R}}
\def\EE{\mathbb{E}}
\def\HH{\mathbb{H}}
\def\SS{\mathbb{S}}

\theoremstyle{definition}
\newtheorem{thm}{Theorem}[section]
\newtheorem{prop}[thm]{Proposition}
\newtheorem{lem}[thm]{Lemma}
\newtheorem{cor}[thm]{Corollary}
\newtheorem{defn}[thm]{Definition}
\newtheorem{exam}[thm]{Example}

\theoremstyle{definition} 
\theoremstyle{definition} 
\theoremstyle{definition}   \newtheorem*{thmfact}{Theorem}
\theoremstyle{definition}   \newtheorem*{conj*}{Conjecture}
\theoremstyle{definition} \newtheorem{conj}{Conjecture}
\theoremstyle{definition}   \newtheorem{ques}[conj]{Question}

\theoremstyle{remark}   
\theoremstyle{remark}

\begin{document}

\begin{frontmatter}

	\title{Rigidity of nonconvex polyhedra with respect to \\ edge lengths and dihedral angles}

\author{Yunhi Cho\fnref{ycho}}
\fntext[ycho]{%\address
	Department of Mathematics, University of Seoul, Seoul 02504 , Korea.
	%\thanks
	Supported by the 2021 sabbatical year research grant of the University of Seoul.
}
\ead{yhcho@uos.ac.kr}

	\author{Seonhwa Kim\fnref{skim}}
\fntext[skim]{	%\address
Department of Mathematics, 	University of Seoul, Seoul 02504, Korea.
%\thanks
Supported  by Basic Science Research Program through the
National Research Foundation of Korea(NRF) funded by the Ministry of Education(2022R1I1A1A01063774).
	}
	\ead{seonhwa17kim@uos.ac.kr}

	\begin{abstract}
		We prove that every three-dimensional polyhedron is uniquely determined by its dihedral angles and edge lengths, even if  nonconvex or self-intersecting, under two plausible sufficient conditions: (i) the polyhedron has only convex faces and (ii) it does not have partially-flat vertices, and under an additional technical requirement that (iii) any triple of vertices is not collinear. The proof is consistently valid for Euclidean, hyperbolic and spherical geometry, which takes a completely  different approach from the argument of the Cauchy rigidity theorem.  Various counterexamples are provided that arise when these conditions are violated, and  self-contained proofs are presented whenever possible. As a corollary,  the rigidity of several families of polyhedra is also established. 
		Finally, we propose two conjectures: the first suggests that Condition (iii) can be removed, and the second concerns the rigidity of spherical nonconvex polygons.		\end{abstract}

\end{frontmatter}

\tableofcontents

\section{Introduction}\label{sec:Intro}
We investigate a rigidity problem for polyhedra in Euclidean,  hyperbolic  and  spherical 3-space with respect to  their dihedral angles and edge lengths, even if they are nonconvex or self-intersecting. We prove that, under a technical assumption prohibiting collinear triple vertices, every polyhedron with only convex faces and no partially-flat vertex is uniquely determined up to congruence by its dihedral angles and edge lengths.
While our approach is elementary, it is somewhat intricate. In particular, the proof relies solely on the combinatorics of the polyhedral graph and  2-dimensional spherical geometry of each vertex's local neighborhood.
Thus, the method is potentially applicable not only to the above three geometries but also to more general situations. For instance, we foresee that our method could be applied to a certain class of Riemannian geometry where the local neighborhood at each point induces a spherical structure, and  totally-geodesic planes with convex polygons are well-defined.
 However, we do not explore in this paper the range of geometries to which our method might be applicable, which we leave as a topic for further study.
 
 \subsection{A historical review}
 It  is  fundamental to study  existence and  uniqueness of polyhedra (or manifold with singularities) under a prescribed geometric condition, which may be the most classical and   contemporary issue in mathematics \cite{fisher_local_2007}. 
Many of  results in this area may be traced back to  Cauchy's rigidity theorem.
 \begin{thmfact}[A. Legendre, A. Cauchy]\label{thm:cauchy}
 	Every convex  polyhedron is uniquely determined by its edge lengths and facial angles\footnote{  Each face of a polyhedron is uniquely determined by its  edge lengths and facial angles.}
 \end{thmfact}
The case of nonconvex polyhedra has been also considered. While R. Connelly found a flexible polyhedron \cite{connelly_counterexample_1977}, it is rare to find flexible polyhedra \cite{gluck_almost_1975}. Recently,
 it has been suggested to study the rigidity of nonconvex polyhedra  under weakly-convex condition \cite{izmestiev_infinitesimal_2010}.
On the other hand, in 1968, Stoker proved the following  with an aid of Cauchy's method 
(for a modern reference,  see \cite{pak_lectures_2009} by I. Pak).
\begin{thmfact}
	[J.J. Stoker \cite{stoker_geometrical_1968}]\label{thm:stoker}
	Every strictly-convex polyhedron is uniquely determined by its dihedral angles and edge lengths.
\end{thmfact}
Furthermore, Stoker also studied this kind of rigidity for a class of nonconvex polyhedra with 4-valent saddle vertices and convex vertices \cite[Theorem III]{stoker_geometrical_1968}.
Our result can therefore be seen as a follow-up to Stoker's original research on nonconvex polyhedra by dihedral angles and edge lengths; however, the approach is completely different. We do not use Cauchy's Arm Lemma or any kind of sign counting to study rigidity.

 Historically, Stoker's  the most influential statement   would be that
\begin{conj*}
	[J.J. Stoker\footnote{In fact, Stoker himself did not state the conjecture regarding hyperbolic polyhedra. However, he is often credited with this conjecture, as people attribute the idea of the rigidity  by dihedral angles to him.}]\label{conj:stoker}
	Every hyperbolic strictly-convex polyhedron is uniquely determined by its dihedral angles. 
\end{conj*}
While a flexible counterexample for the spherical case has been discovered by J.-M. Schlenker \cite{schlenker_dihedral_2000} and the local rigidity of the hyperbolic case has been proved  by R. Mazzeo and G. Montcouquiol \cite{mazzeo_infinitesimal_2011}, the conjecture still remains unsolved, to the best of the authors' knowledge.
Intriguingly, it has recently turned out that the Stoker conjecture, a fundamentally classical subject, is related to the volume conjecture, a modern topic associated with quantum representations \cite{belletti_volume_2023}.
On the other hand, one may also consider a rigidity of polyhedra given only edge lengths and not dihedral angles. 
This leads to a problem of rigidity for \emph{frameworks}, which are another generalization of polyhedra, and this topic has also been actively studied (for example, see \cite{roth_rigid_1981}). 

Lastly, we would like to note that the moduli space of geometric polyhedra with  a fixed combinatorial type  can be embedded in the Cartesian product of edge lengths and dihedral angles when the rigidity holds.
In this parametrization, the famous Schl\"afli differential formula, given by
$ \pm 2 d V = \sum_{e\in E} l_e d\theta_e,$ 
holds and is a Liouville 1-form in the moduli space. Here, $V$ is the volume of the polyhedron, $l_e$ is the length of edge $e$, and $\theta_e$ is the dihedral angle at edge $e$.
This formula provides a reason why studying dihedral angles and edge lengths can be interesting and potentially helpful for further research connecting polyhedral combinatorics and non-Euclidean geometry. 

 \subsection{Definitions and terminology}\label{sec:defterm}
In this article, the term `geometry' refers to Euclidean, hyperbolic,  or spherical geometries, and our main objects of study are 3-dimensional geometric polyhedra. Before presenting our results, we  briefly clarify what we mean by polyhedra, taking nonconvex cases into account. This is because a rigorous definition of nonconvex polyhedra varies slightly depending on the mathematical context or research area.

A \emph{combinatorial polyhedron} or \emph{polyhedral graph} is defined as a $3$-connected simple planar graph\footnote{In this context, `simple graph' refers to a graph without 1-cycles or 2-cycles. `3-connected' denotes that no set of fewer than three vertices can separate the graph into disconnected components.} embedded on the oriented 2-sphere.
According to Steinitz's theorem, this combinatorial structure is equivalent to the vertex-edge-face structure arising from the boundary of the convex hull formed by a finite set of vertices in Euclidean 3-space.\!\footnote{Steinitz's theorem can be applied without any issues in hyperbolic geometry. However, caution is required when dealing with spherical geometry due to the issue of convexity. In this paper, we impose the requirement that every convex polyhedron should be contained within the interior of a half-space. This additional condition, which is discussed in the latter part of this section, ensures the universality of Steinitz's theorem across all three geometries.}

A \emph{geometric polyhedron} or simply a \emph{polyhedron} is a (possibly self-intersecting) continuous image of a combinatorial polyhedron in a geometric 3-space (i.e., Euclidean 3-space $\EE^3$, hyperbolic 3-space $\HH^3$, or  spherical 3-space $\SS^3$) with the restriction that each face is embedded into a totally geodesic plane and each edge is embedded into a geodesic.  Since the 2-sphere of the domain is oriented and each face is a topological disk in our setting, dihedral angles and  facial angles are defined naturally.
When we say that entities like vertices or faces are \emph{collinear} or \emph{coplanar} in spherical or hyperbolic space, it means that they lie on a common geodesic or a totally geodesic plane, analogous to Euclidean space.

Similarly, a (geometric) \emph{polygon} is also defined as a continuous image of a (combinatorial) polygon into a geometric space, with the restriction that each edge is embedded as a geodesic. 
When we are considering a polygon on an oriented totally geodesic surface, we can distinguish between the left side and the right side at each point on the polygon. Once  a side is fixed,   \emph{interior angles} are defined unambiguously, even if the polygon is self-intersecting.

We may use the terms `polyhedron' and `polyhedral graph' interchangeably without making a clear distinction unless it would cause confusion.
A geometric polyhedron is sometimes referred to as a \emph{geometric realization} or simply a \emph{realization} of a polyhedral graph, and a polyhedral graph is sometimes referred to as  the \emph{combinatorial type} of a geometric polyhedron.

Two polyhedra are considered to be \emph{the same} or \emph{congruent} if there exists an isometry $\Phi$ of $\EE^3$, $\SS^3$, or $\HH^3$ that maps one realization $\phi_1$ to another realization $\phi_2$. Note that this doesn't necessarily imply that $\Phi \circ \phi_1 = \phi_2$ as a function due to  reparametrization issue, but it does mean that the images of $\Phi \circ \phi_1$ and $\phi_2$ just coincide on each face. Let $\V(P)$ denote the set of vertices of $P$. It is obvious that the following three conditions are equivalent: (i) two geometric realizations $\phi_1$ and $\phi_2$  are congruent; (ii) $\phi_1(\V(P))=\phi_2(\V(P))$; (iii) all corresponding edge lengths, facial angles, and dihedral angles are the same.

 If all realizations of a polyhedron or polygon under certain constraints, such as fixed edge lengths, always result in congruent polyhedra or polygons, then we say that the polyhedron is uniquely determined, or \emph{rigid}, by the constraints. For example, we say that every triangle is uniquely determined by its three edge lengths ($SSS$\footnote{This is a term commonly used in the context of triangle congruence in school mathematics. It is worth noting that there exist certain issues associated with  $SSS$, $ASA$, $SAS$, and $AAA$ congruence conditions if you consider nonconvex spherical polygons: these will be discussed in Section~\ref{sec:trianglerigidity}.}) in Euclidean or hyperbolic  geometries.

 The term \emph{convex}, when applied to a polygon or polyhedron, indicates that (C1) it has no self-intersections, (C2) its interior angles or dihedral angles are all less than or equal to $\pi$, (C3) it is contained entirely within the interior of a half-space (or hemisphere in the case of spherical geometry).
 The term \emph{strictly-convex} implies that it is convex and its angles are less than $\pi$. 
 In Euclidean space, it is well known that every convex polygon or polyhedron defined by the above three conditions is equivalent to the convex hull of its vertices or the intersection of half-spaces (see, for example,  \cite[Section 1]{alexandrov_convex_2005}).
 This equivalence is also clearly valid for  $\H$ when using the Kleinian model.
 Condition (C3) is trivially satisfied for Euclidean and hyperbolic geometry since our polygon or polyhedron is always compact from the definition.
 In spherical geometry, additional consideration is needed. Note that every polygon satisfying conditions (C1) and (C2) is contained within a hemisphere due to Theorem~\ref{convexsphpoly}. Condition (C3) is necessary to exclude cases like bigons, which appears in Example~\ref{sphBigon}. In particular, under condition (C3), any geometric realization of every combinatorial polyhedron occurs simultaneously in all three geometries,   by using the Kleinian model for hyperbolic geometry and the gnomonic projection for spherical geometry \cite{ratcliffe_foundations_2006}.
Finally, we remind about bigons in our setting for spherical geometry. Bigons can be considered spherical polygons but they can never be convex due to condition (C3). Regardless of whether a spherical polyhedron is convex or nonconvex, bigon faces are not allowed due to the requirement that the 1-skeleton is a polyhedral graph.

 \subsection{Main results}
 
 Let an angle be  \emph{singular} if it is $0$ or $\pi$.
Let an edge $e$ of a polyhedron be \emph{flat} if its dihedral angle is singular, i.e. the adjacent two faces of $e$ are coplanar.
 Let a vertex $v$ be  \emph{flat} if all adjacent edges are flat, i.e., all adjacent faces of $v$ are coplanar.
 Let a vertex $v$ be  \emph{partially-flat} if some faces adjacent to a vertex $v$ are coplanar. Note that a vertex may be partially-flat  even though  no adjacent edge is flat. It is clear that existence of a flat edge implies the adjacent vertices are partially-flat.
 Our main result is stated  as follows:
 \begin{thm}\label{mainThm}
 	If a geometric, possibly nonconvex and self-intersecting, polyhedron satisfies that (i) every faces is convex, (ii) no vertices  are partially-flat and (iii) any triple of vertices is not collinear, then the polyhedron is uniquely determined by its dihedral angles and edge lengths.	
 \end{thm}
 
 Remark that every polyhedron that satisfies these conditions always has strictly-convex faces; otherwise, there would be collinear triples of vertices.  
 
 A polyhedron is considered \emph{weakly-convex} if its vertex set is the same as the vertex set of another strictly-convex polyhedron. Under the weakly-convex condition, there must be no nonconvex faces or collinear triples of vertices. As a result, we have an immediate corollary:
 
 \begin{cor}
 	If a polyhedron is weakly-convex and has no partially-flat vertices, then it is uniquely determined by its dihedral angles and edge lengths.
 \end{cor}

Upon careful consideration of the proof, we  easily derive a statement in which Condition (ii) of Theorem~\ref{mainThm} is replaced as follows:
\begin{thm}\label{thm2}
	 	If a geometric, possibly nonconvex and self-intersecting, polyhedra satisfies that (i) every faces is convex, (ii) there are no flat edges and no set of seven vertices lies on the same plane, and (iii) any triple of vertices is not collinear, then the polyhedron is uniquely determined by its dihedral angles and edge lengths.	
\end{thm}

In fact, Condition (iii) in Theorem~\ref{mainThm} and Theorem~\ref{thm2} is an auxiliary condition to bypass intricate and delicate situations that arise during  proof.
We expect that it can be removed in subsequent studies, and we currently conjecture as follows:
\begin{conj}
	Every geometric, possibly  nonconvex and self-intersecting, polyhedron is  uniquely determined by its dihedral angles and edge lengths, 
	if it satisfies 	 that (i) every faces is convex, (ii) no vertices are partially-flat.
	\end{conj}
	This conjecture is indeed expected to be very challenging to prove, making it a long-term goal. One possible intermediate milestone could be to relax Condition (iii), specifically the statement  that no vertices of the polyhedra overlap in position, rather than eliminating Condition (iii) entirely.
	Another potential intermediate step is to substitute the convex face condition of Condition (i) with  strictly-convex face condition to simplify the problem. These relaxations are expected to reduce complexity involved in analyzing the configurations.

It is worth noting that even though a polyhedron is uniquely determined by its dihedral angles and edge lengths, it does not necessarily satisfy our three conditions. Therefore, there are possibilities for other types of conditions that ensure uniqueness, which is also another topic for further study.

The most complicated and tedious part of our study is the proof of the rigidity of nonconvex spherical quadrilaterals as stated in Theorem~\ref{thm:sphquadRigid}. We did not investigate the generalization to $n$-gons with $n>4$, since it is not necessary for the study of polyhedra. However, it could be a subject of further study in its own right. In this regard, we propose the following conjecture, which is a nonconvex generalization of Theorem~\ref{rigiditySphPolygon}:

\begin{conj}\label{conj:polygon}
	Consider a spherical $n$-gon, which may be nonconvex and self-intersecting. Assume that the $n$-gon satisfies the following conditions: (i) all edge lengths are less than $\pi$, and (ii) any pair of two edges are not collinear.
	Under these conditions, the spherical $n$-gon is uniquely determined by $n-3$  edge lengths and all interior angles.
\end{conj}

Remark that all interior angles are non-singular due to condition (ii). It's worth noting that a counterexample exists if condition (ii) is replaced with the condition that prohibits singular interior angles. See Example~\ref{ex:deg6flex} and \ref{sph-quad}. For a counterexample that may arise when condition (i) is absent, refer to Example~\ref{sph-quad} and Example~\ref{nonconvexQuad}.

Finally, although we present a counterexample regarding  partially-flat vertices  in Example~\ref{ex:partially-flat-vertex},  it does not break  local rigidity, which is the property of not allowing an one-parameter family of deformations. We have a question about this issue as follows.
\begin{ques}\label{ques}
	Consider a polyhedron with convex faces and no flat edges. Is it flexible or locally rigid with respect to its dihedral angles and edge lengths?	
\end{ques}

 \subsection{Outlines}

 In Section~\ref{sec:convexcase}, we  first demonstrate that strictly-convex polyhedra can be uniquely determined by their dihedral angles and edge lengths. Although this result was previously established by J. J. Stoker in \cite{stoker_geometrical_1968}, we offer an alternative proof which can be generalized to a broader range of nonconvex cases beyond Stoker's original work.  
In essence, the new proof strategy involves reducing vertices of a given polyhedron.  Every polyhedron possesses a `rigid' vertex, defined combinatorially. Since the polyhedron is strictly-convex, the neighborhood of the vertex is uniquely determined by dihedral angles and edge lengths, which enables us to reduce it while preserving the geometric information. We refer to this process as `vertex reduction', a central concept of this paper. All other components of this study may be viewed as efforts to generalize this vertex reduction process to nonconvex situations. Although such generalization introduces certain complexities, the essential idea in the strictly-convex case remains unchanged. 

In Section~\ref{counterexamples}, we present various counterexamples that illustrate why we exclude nonconvex faces, flat edges, and partially-flat vertices. 

In Section~\ref{generalizedNonCon}, we describe the strategy for generalizing our method to nonconvex polyhedra and redefine the vertex reduction procedure  in a rigorous combinatorial way. 
Remark that the inclusion of Condition (iii), which states that there should be no triple of collinear vertices, is to prevent the occurrence of degenerate triangles during vertex reduction. This is crucial not only because degenerate faces are prohibited by the definition of a geometric polyhedron, but also because they would introduce significant technical complexities during vertex reduction.
It is worth noting that singular angles such as 0 or $\pi$ can arise during vertex reduction, but they only occur at newly-added edges, which are always adjacent to a triangle. Therefore, we can handle these singularities.

In Section~\ref{nonconSph}, we study a rigidity of nonconvex spherical polygons. For the reader's convenience, we also provide a proof of spherical trigonometry beyond strictly-convex triangles in Section~\ref{sphTrigonometry}. 
Although trigonometry is a classical topic, some non-trivialities arise in spherical geometry that do not occur in Euclidean or hyperbolic trigonometry. While most textbooks and literature only consider trigonometry for strictly-convex triangles, the trigonometric rules actually hold even without such assumption.
Although this fact is certainly well-known, we were unable to find recent literature on such `generalized trigonometry', apart from a book nearing a century old \cite{klein_ubersicht_1933}.
Based on this study of spherical trigonometry, we establish  rigidity theorems for spherical triangle (3-gon) and quadrilateral (4-gon), which are crucial for proceeding with the induction through the vertex reduction process.

\section{Rigidity for strictly-convex polyhedra}\label{sec:convexcase}
Let us begin with basic definitions. The \emph{degree}, denoted as $\deg(v)$ for a vertex $v$ or $\deg(f)$ for a face $f$, refers to the number of adjacent edges.
\begin{defn}
	The \emph{non-triangular degree} of a vertex $v$, denoted by $\tau (v)$, is the number of adjacent non-triangular faces at the vertex $v$. We say that a vertex $v$ is \emph{rigid} if $\tau(v) \leq 3$. 
\end{defn}
\begin{defn}\label{def:Nv}
	Let $P$ be a geometric polyhedron and $v$ be a vertex of $P$.
	The \emph{local triangulation} of $P$ at $v$, denoted by $P_v$, is the polyhedron obtained from $P$ by adding an edge to each non-triangular face adjacent to $v$, thereby ensuring that there are no non-triangular faces adjacent to $v$, as illustrated in Figure~\ref{fig:Nv}.
	The \emph{triangularized neighborhood} of $v$, denoted by $\N_v$, is the simplicial star of $v$ in $P_v$, comprising all faces  adjacent to $v$ in $P_v$.
	Let $\bdN_v$ denote the boundary of $\N_v$. 
	\end{defn}
\begin{figure}[H]
\begin{center}
	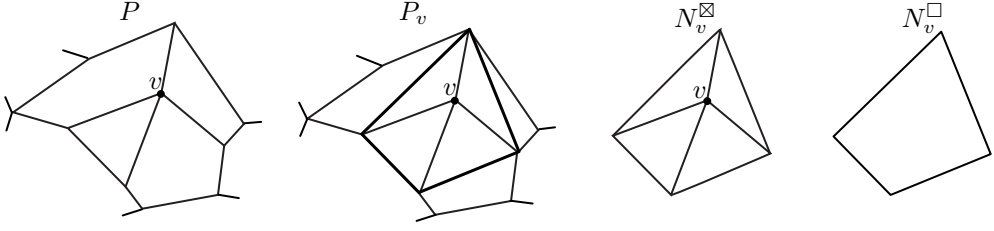
	\caption{Local triangulation $P_v$ and the triangularized neighborhood $\N_v$ of $v$}
	\label{fig:Nv}
\end{center}
\end{figure}	
The boundary $\bdN_v$  is obviously an $n$-gon for $\deg(v)=n$ and it may not  be a plane polygon. 
\begin{defn}
	A \emph{vertex figure} at $v$, denoted by $S_v$, is defined as the spherical polygon on the unit sphere  obtained by intersecting the polyhedron $P$ with a sufficiently small sphere centered at the vertex $v$ and then scaling up the spherical polygon and the sphere to have radius one. 
\end{defn}	
	\begin{figure}[h!]
		\begin{center}
			\input{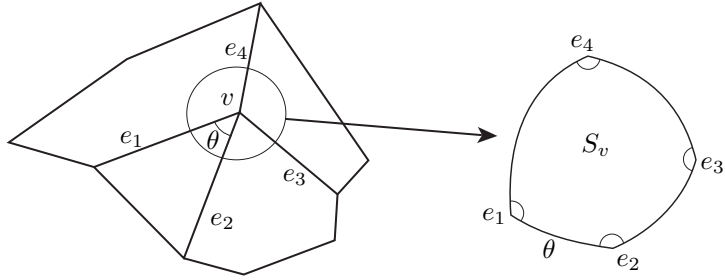}
			\caption{The local configuration near $v$ and the vertex figure $S_v$, where the facial angle $\theta$ corresponds to the edge lengths of $S_v$}
			\label{fig:vertexfigure}
		\end{center}
	\end{figure}	
	Note that the dihedral angles and facial angles adjacent to $v$ correspond to interior angles and edge lengths  of $S_v$, respectively. 
The following proposition demonstrates why we use the term `rigid' for such a vertex.
\begin{prop}
	\label{rigidvertexthm}
Let $P$ and $Q$ be two strictly-convex polyhedra of the same combinatorial type.
If a vertex $v$ of the polyhedral graph is rigid, and the lengths of all corresponding adjacent edges, along with their dihedral angles in $P$ and $Q$, are the same, then the corresponding neighborhoods $\N_v$ in $P$ and $Q$ are isometric to each other.
\end{prop}

\begin{proof}
Consider the vertex figure $S_v$ at $v$. If $\deg(v)=n$, then $S_v$ forms a strictly-convex spherical $n$-gon. Since $v$ is rigid, it is adjacent to at least $n-3$ triangular faces. As the lengths of all edges of these triangular faces are known, all their interior angles are also uniquely determined.
	Recall that these facial angles  correspond to the edge lengths  of $S_v$. 
	 Thus, we  determine  at least $n-3$ edge lengths in $S_v$. 
	Since	$S_v$ has at most three unknown edge lengths, the spherical $S_v$ is unique up to congruence  by Theorem~\ref{rigiditySphPolygon}. All edge lengths of $P$ are already known. It is obvious that the triangularized neighborhood  $\N_v$ in $P$ and $Q$ are isometric to each other.  
\end{proof}

In fact, this proposition is an obvious   corollary of Theorem~\ref{rigiditySphPolygon}, which we will prove below,  regarding the rigidity of strictly-convex spherical polygons.
Before proceeding with the proof of Theorem~\ref{rigiditySphPolygon}, some preliminary work is required.
 Recall the definition of  convex spherical polygons presented in Section~\ref{sec:defterm}. According to condition (C3) of this definition, a convex spherical polygon is contained within a hemisphere automatically.  We can, however, prove that this containment is actually a consequence of conditions (C1) and (C2) even without invoking condition (C3).
 Although a similar matter has been previously addressed in \cite[Section 2.7]{stoker_geometrical_1968}, we provide our own proof  for the convenience of the reader.  It also lays the foundation for important steps in this paper.
 
For convenience, let us introduce  notation that will be used throughout this paper. The polygon with vertices $a$, $b$, $c$ and $d$ is denoted by $\polygon{a\. b\. c \.d}$, and the polygonal path passing through these vertices in order is denoted by $\edgepath{a\. b\. c\. d}$. Also, $\vect{a\.b}$ represents a geodesic extended in the direction from $a$ to $b$.
 
\begin{thm}\label{convexsphpoly}
	Let $P$ be a spherical polygon on a 2-sphere. If $P$ does not have any self-intersection and all its interior angles are less than $\pi$, then $P$ is contained in a hemisphere~$H$. In particular, if an edge $e$ is lying on the great circle $\partial H$, then $P \ssetminus e$ is contained within the interior of $H$. Therefore, $P$ is either a bigon or is contained within the interior of a hemisphere.
\end{thm}
\begin{proof}
	Every bigon with an interior angle less than $\pi$ is obviously contained in a hemisphere. 
	For an $n$-gon $P$ with $n \geq 3$, let us consider the lengths of two consecutive  edges. 
	At least one of these edges must have a length less than $\pi$; otherwise, there would be an intersection for a pair of antipodes, violating the non-self-intersecting condition. 
	Consequently, there must exist an edge $e$ with  length less than $\pi$, where  the endpoints $\partial e$ are denoted by $x$ and $y$.
	Let us take a hemisphere $H$ such that the boundary circle contains the edge $e$ and the two adjacent edges at  $x$ and $y$  go into the hemisphere $H$, as illustrated in Figure~\ref{fig:convexSphericalPolygon}.
\begin{figure}[h!]
	\begin{center}
		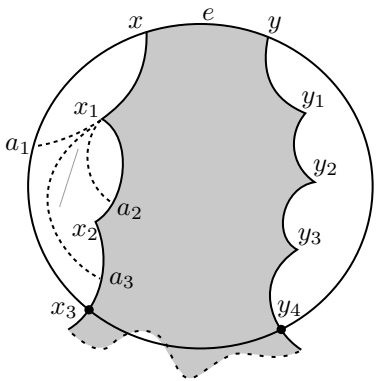
		\caption{A strictly-convex polygon that is not entirely contained within the hemisphere}
		\label{fig:convexSphericalPolygon}
	\end{center}
\end{figure}	

Let us suppose, contrary to the conclusion, that $P \ssetminus e$   does not contained in the interior of $H$. Some part of $P$ must intersect $\partial H \ssetminus e$. Let us examine the left-most intersections with $\partial H \ssetminus e$. 
For example, consider the shortest edge path $\edgepath{x\, x_1\. x_2\. x_3}$ on $\partial P$ which meets $\partial H \ssetminus e$ in Figure~\ref{fig:convexSphericalPolygon}. 
The edge path $\edgepath{x\, x_1\. x_2\. x_3}$ and $\edgepath{x_3\. x}\subset \partial H$ forms a simple closed curve~$C$. 
Similarly, for the right-most intersection, we can find another disjoint simple closed curve $D= \edgepath{y\,y_1\. y_2\. y_3 y_4}\hspace{0.1em}\cup\hspace{0.1em} \edgepath {y_4\. y}$. 
Now, we claim that these $C$ and $D$ must contain a bigon respectively. 
To see this, for example,  let us consider $\vect{x\, x_1}$ which  is extended $\edgepath{x\, x_1}$ from $x_1$ geodesically. Then  $\vect{x\, x_1}$ must intersect $C$ at either $a_1$, $a_2$, or $a_3$; otherwise, $C$ would enclose a great circle inside $H$, which is impossible.  If the intersection point is $a_1$ or $a_2$ then the bigon appear immediately. If it is $a_3$, let us consider  a smaller simple closed curve $C'=\edgepath{x_1\.x_2\.a_3\. x_1}$ with fewer number of edges. We can find a bigon in $C'$ by induction on the number of edges and it proves the claim. Consequently, the hemisphere $H$ must contain at least two disjoint bigons from $C$ and $D$ respectively and it is impossible in 2-dimensional spherical geometry. 
Finally, since the length of  $e$ is less than $\pi$, there is a pair of antipodes in $\partial H \ssetminus e$. By rotating the hemisphere $H$ slightly around the axis defined by the pair of antipodes, we obtain that the polygon $P$ lies entirely within the interior of this rotated hemisphere. This concludes the proof.
\end{proof}

By the fact that the lengths of all geodesic segments contained in the interior of a hemisphere is always less than $\pi$, we have a corollary of Theorem~\ref{convexsphpoly} as follows.
\begin{cor}\label{lengthSphPolygon}
	Every edge length of a  spherical strictly-convex polygon is less than $\pi$.
\end{cor}

Remark that these results also hold for three-dimensional polyhedra. If there are no self-intersections and all facial and dihedral angles are less than $\pi$, the polyhedron $P$ is contained within a three-dimensional hemisphere $H$. Suppose otherwise, one could consider a two-dimensional section intersecting with $P$, which forms a strictly-convex polygon that cannot be contained in a hemisphere. This leads to a contradiction to Theorem \ref{convexsphpoly}

Now we consider a rigidity theorem about  spherical strictly-convex polygons as follows.

\begin{thm}\label{rigiditySphPolygon}
	Every spherical strictly-convex $n$-gon $P$  is uniquely determined by all of its interior angles and $n-3$ of its edge lengths.
\end{thm}

\begin{proof}	
	If $P$ is a triangle, it is uniquely determined by its three angles\footnote{
		If an interior angle or an edge length of a spherical triangle is equal to $\pi$, then the triangle is not uniquely determined by its angles alone. We exclude this situation due to the strictly-convex condition.} 
	by the spherical cosine rule, without any edge lengths.
	We use induction on $n$. For every strictly-convex $n$-gon with $n\geq 4$, there must be an edge $e$ whose length is determined, because at least $n-3$ edge lengths are given. Refer to the $4$-gon $P=\polygon{w_1\.w_2\.w_3\.w_4}$ in Figure~\ref{fig:polygonreduction}, where the length of $e=\edgepath{w_1\.w_2}$ is given. Note that two interior angles adjacent to the edge $\edgepath{w_1\.w_2}$ are given. As we extend two adjacent edges $\edgepath{w_4\.w_1}$ and $\edgepath{w_3\.w_2}$ of $\edgepath{w_1\.w_2}$ in the outward direction of the convex polygon, the two lines meet at a new vertex $u$ and the edge $\edgepath{w_1\.w_2}$ is removed, as illustrated in Figure~\ref{fig:polygonreduction}. 
	\begin{figure}[h]
		\begin{center}
				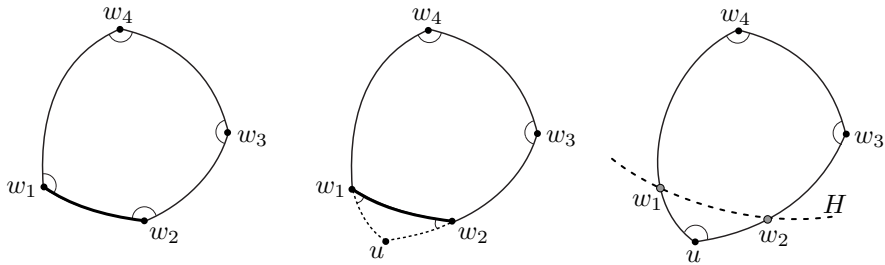
				\caption{A spherical $4$-gon $P$ given with  all interior angles and one edge length and the reduced $3$-gon $P'$}
				\label{fig:polygonreduction}
			\end{center}
	\end{figure}

	 We claim that the reduced $(n-1)$-gon $P'=\polygon{u\,w_3\.w_4} = P \cup\, \polygon{u\, w_2\.w_1}\,$ is strictly-convex. The edge $\edgepath{w_1w_2}$ has  length less than $\pi$ by Corollary~\ref{lengthSphPolygon}.	 
	  The interior angle at vertex $u$ is less than $\pi$ by the spherical sine rule\footnote{
	  	In most literature, the Sine rule is established under the strictly-convex condition. Therefore, it cannot be employed to prove that the interior angle $u$ is less than $\pi$.  We will prove the  trigonometry without angle assumption in Section~\ref{sphTrigonometry}.}, and hence all interior angles of $P'$ are less than $\pi$. The polygon~$P$ is contained in  a hemisphere $H$ whose boundary contains $\edgepath{w_1 w_2}$ by Theorem~\ref{convexsphpoly} and  the 3-gon $\polygon{u\,w_2\.w_1}$ is contained in the opposite hemisphere. Therefore, $P'$ does not have any self-intersection and it is obviously not a bigon. We  conclude that $P'$ is strictly-convex, satisfying the definition in Section \ref{sec:defterm}. 
	 Now, all angles of $P'$ are given and the number of unknown edge lengths of $P'$ does not increase from $P$. By proceeding with the induction step, we  finally arrived at a triangle and this completes the proof.
\end{proof}

Now, let us recall a well-known lemma that has appeared many times in the literature (for example, see p.237 of \cite{grunbaum_convex_2003}). Let $V$,$E$, and $F$ denote  the number of vertices, edges and faces, respectively.  Let $V_i$ and $F_i$ denote the number of vertices and faces of degree $i$, respectively. 
\begin{lem}
	\label{lem:v3f3} 
	For every polyhedral graph $P$, we have
	$$ V_3 + F_3 = \sum_{n\geq5} (n-4)(V_n + F_n) + 8.$$
\end{lem} 
\begin{proof}
	Each edge is 
	adjacent to  vertices and faces exactly twice respectively. Hence we get  
	\begin{equation}\label{eq:2ev3f3}
		2E=\sum_{n\geq3}n V_n =\sum_{n\geq3}n F_n.
	\end{equation}
	
	Recall the Euler's formula $V - E + F =2$ and the following completes the proof.
	\begin{align*}
		V_3+F_3 &=   4E - 4V - 4F + V_3 + F_3 +8 \\
		&=   \sum_{n\geq3}n V_n + \sum_{n\geq3}n F_n -4\sum_{n\geq3}V_n  - 4\sum_{n\geq3}F_n +V_3 +F_3 +8 \\
		&= \sum_{n\geq5} (n-4)V_n + \sum_{n\geq5} (n-4)F_n + 8
	\end{align*}
\end{proof}

The following is a simple yet key observation  which is a straightforward consequence of  Lemma~\ref{lem:v3f3}.
\begin{lem}\label{lem:existencerigid}
	For every  polyhedral graph $P$, there always exists a rigid vertex.
\end{lem}
\begin{proof}
	Suppose there are no rigid vertices.
	Then it is obvious that $ V_3 =0 $ and  all vertices must meet triangle faces at most $\deg (v) -4 $. Therefore we obtain
	$	3F_3 \leq \sum_{n\geq 5} (n-4) V_n$, which  contradicts to 	$ F_3 \geq \sum_{n\geq 5} (n-4)V_n + 8 $ established by Lemma~\ref{lem:v3f3}.
\end{proof}

Then we reprove  Stoker's theorem  using a completely novel method, as follows.
\begin{thm}\label{convexRigidity}
	Every strictly-convex polyhedron $P$ is uniquely determined by its dihedral angles and edge lengths.
\end{thm}

\begin{proof}
	We  use induction on the number of vertices.
	By Lemma~\ref{lem:existencerigid}, we know there exists at least one vertex $v$ for any strictly-convex polyhedron. 
	By Proposition~\ref{rigidvertexthm}, the $\N_v$ is uniquely determined up to isometry. 
	Furthermore,  we can get the polyhedron $P'$ as the convex hull\footnote{		
	In spherical geometry, we need to be careful as there can be multiple geodesics connecting two points. By considering the strictly-convex condition stated in the theorem, the polyhedron is contained in the interior of a hemisphere and we can utilize the convex hull by selecting the unique shortest geodesic segments.} of $\V(P) \ssetminus v$ from the previous polyhedron $P$ as in Figure~\ref{fig:removingrigidvertex}, where  
	the triangularized neighborhood $\N_v$ consists of five vertices $v$, $a$, $b$, $c$ and $d$. 
	\begin{figure}[H]
		\begin{center}
			\def\svgwidth{0.8\columnwidth} 
			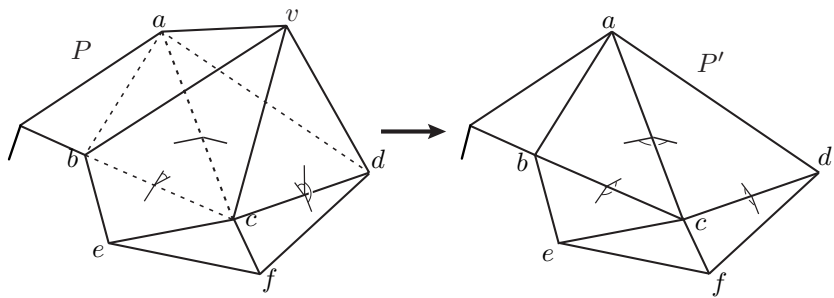
			\caption{
				$P'$ is the vertex reduction of $P$ at a rigid vertex $v$. 
			}
			\label{fig:removingrigidvertex}
		\end{center}
	\end{figure}
Note that all edge lengths and dihedral angles of $P'$ are entirely determined by $P$ and $\N_v$. This is because the position of the vertices of $\N_v$ is determined up to isometry according to Proposition~\ref{rigidvertexthm}, and the lengths and dihedral angles of any newly-added edges can be computed based on these vertex positions. 
As a result, the polyhedron $P'$ has fewer vertices due to removing one rigid vertex, is strictly-convex due to being obtained via the convex hull, and has all its edge lengths and dihedral angles determined.

Through induction, we  reach the case of $\V(P)=\V(\N_v)$ where $P$ is uniquely determined up to congruence since the vertex positions are determined up to isometry. Beyond this point, we don't need to proceed further. 
This should happen because the reduction process, which decreases the number of vertices, eventually leads to a tetrahedron.
In conclusion, by considering the reverse of the induction process, i.e., adding vertices one by one using the information of the previously removed $\N_v$'s, we can uniquely reconstruct the original polyhedron $P$.
\end{proof}

The most crucial part of the proof is the process of reducing the number of vertices to obtain a strictly-convex polyhedron. This concept, which we refer to as \emph{vertex reduction}, also plays a significant role when generalizing the theorem to the case of nonconvex polyhedra.

\section{Counterexamples}\label{counterexamples}

Before examining the case of nonconvex polyhedra, let us first consider various counterexamples where uniqueness does not hold when regarding nonconvex polyhedra. Through this, we obtain a list of necessary conditions for nonconvex polyhedra such that uniqueness is guaranteed by their dihedral angles and edge lengths.

First of all, in  spherical geometry, we need to decide how to handle bigon faces.  Let's keep in mind the following example that has a bigon as a face.
\begin{exam}[spherical  bigons]\label{sphBigon}
	Let's consider a shape in spherical geometry that we'll call a `polygonal-hedron'. This shape has only two vertices in antipodal positions and exclusively comprises bigon faces. It resembles the topological suspension of a polygon.
	  Notably, if the polygon has more then three edges, the polygonal-hedron  is flexible since the spherical polygon of the vertex figure can be continuously deformed while preserving its interior angles. This example illustrates that even if a compact region in $S^3$ is defined by the intersection of half-spaces, its 1-skeleton may not be 3-connected. Moreover, the convex hull of its vertices is not well-defined and, even if the definition were to be established, it would not coincide with the original shape of the polygonal-hedron.
\end{exam}
In Section~\ref{sec:defterm}, we established the condition that the 1-skeleton of a polyhedron should be a polyhedral graph. Thus, in this paper, the polygonal-hedron mentioned in the previous example is not considered as a polyhedron.

\begin{exam}[flat vertex]
If we allow flat vertices, the problem fundamentally becomes a two-dimensional matter, through which we can construct counterexamples. We first identify non-congruent plane graphs with fixed corresponding edge lengths while preserving the boundary polygon $B$. Next, consider a polyhedron $P$ with a face that is isometric to the polygon $B$. If the non-congruent plane graphs are plugged into $P$ along the boundary $B$, these produce counterexamples with flat vertices.
 If nonconvex faces are also allowed, it is easy to construct an 1-parameter family, as demonstrated in Figure~\ref{fig:nonconvexFaceFlatvertex}, where  edge lengths are fixed during flexing and all vertices inside $B$ are flat and  dihedral angles are all fixed at $\pi$.

	\begin{figure}[H]
	\begin{center}
		%% Creator: Inkscape 1.2.2 (732a01da63, 2022-12-09), www.inkscape.org
%% Accompanies image file 'flatVertexNonconvexFace.pdf' (pdf, eps, ps)
\begingroup%
  \makeatletter%
  \providecommand\color[2][]{%
    \errmessage{(Inkscape) Color is used for the text in Inkscape, but the package 'color.sty' is not loaded}%
    \renewcommand\color[2][]{}%
  }%
  \providecommand\transparent[1]{%
    \errmessage{(Inkscape) Transparency is used (non-zero) for the text in Inkscape, but the package 'transparent.sty' is not loaded}%
    \renewcommand\transparent[1]{}%
  }%
  \providecommand\rotatebox[2]{#2}%
  \newcommand*\fsize{\dimexpr\f@size pt\relax}%
  \newcommand*\lineheight[1]{\fontsize{\fsize}{#1\fsize}\selectfont}%
  \ifx\svgwidth\undefined%
    \setlength{\unitlength}{94.04563324bp}%
    \ifx\svgscale\undefined%
      \relax%
    \else%
      \setlength{\unitlength}{\unitlength * \real{\svgscale}}%
    \fi%
  \else%
    \setlength{\unitlength}{\svgwidth}%
  \fi%
  \global\let\svgwidth\undefined%
  \global\let\svgscale\undefined%
  \makeatother%
  \begin{picture}(1,0.83589606)%
    \lineheight{1}%
    \setlength\tabcolsep{0pt}%
    \put(0,0){\includegraphics[width=\unitlength,page=1]{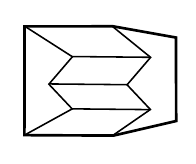}}%
    \put(0.17747789,0.4105764){\color[rgb]{0,0,0}\makebox(0,0)[lt]{\lineheight{1.25}\smash{\begin{tabular}[t]{l}$x$\end{tabular}}}}%
    \put(0.01199751,0.53686486){\color[rgb]{0,0,0}\makebox(0,0)[lt]{\lineheight{1.25}\smash{\begin{tabular}[t]{l}$B$\end{tabular}}}}%
    \put(0.69999083,0.39805733){\color[rgb]{0,0,0}\makebox(0,0)[lt]{\lineheight{1.25}\smash{\begin{tabular}[t]{l}$y$\end{tabular}}}}%
    \put(0,0){\includegraphics[width=\unitlength,page=2]{flatVertexNonconvexFace.pdf}}%
  \end{picture}%
\endgroup%

		\caption{
	A flexible pattern within the fixed hexagonal boundary $B$ (indicated by a bold line), where  the edge $\protect\edgepath{x\.y}$ can move horizontally}
	\label{fig:nonconvexFaceFlatvertex}
\end{center}
\end{figure}

\end{exam}

\vspace{-1em}
Even though we require each face to be strictly-convex, we can create a non-congruent pair of polyhedra using flat vertices, as follows.
	\begin{figure}[H]
	\begin{center}
		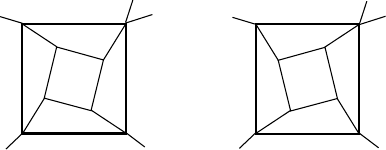
		\caption{
			A non-congruent pair with flat vertices and convex faces, where  all corresponding edge lengths are the same, and all dihedral angles in the bold square are $\pi$}
		\label{fig:convexFaceFlatvertex}
	\end{center}
\end{figure}

Therefore, we  exclude any nonconvex faces and flat vertices to ensure  unique realization. Nevertheless, there is still a counterexample even when all faces are strictly-convex and there are no flat vertices, provided that a flat edge is present, as demonstrated below.

\begin{exam}[flat edge]	\label{flatedge}
This example is a modification of the polyhedra in Figure~\ref{fig:convexFaceFlatvertex}.
More specifically, consider two regular cubes with six quadrilateral faces. Draw a flat, twisted prism graph on each face, as shown in Figure~\ref{fig:nonconvex_pi_angle}. We can choose two different twisting directions: left and right. These directions ensure that the corresponding edges are of the same length and the dihedral angles are $\pi$.
To avoid flat vertices, we attach a quadrilateral pyramid to the internal quadrilateral on each face. The resulting nonconvex polyhedra cannot be congruent to each other if we choose different twisting directions on each face, which prevents any symmetry of the polyhedra.
	\begin{figure}[H]
		\begin{center}
			%		\scriptsize
			\includegraphics[width=8cm]{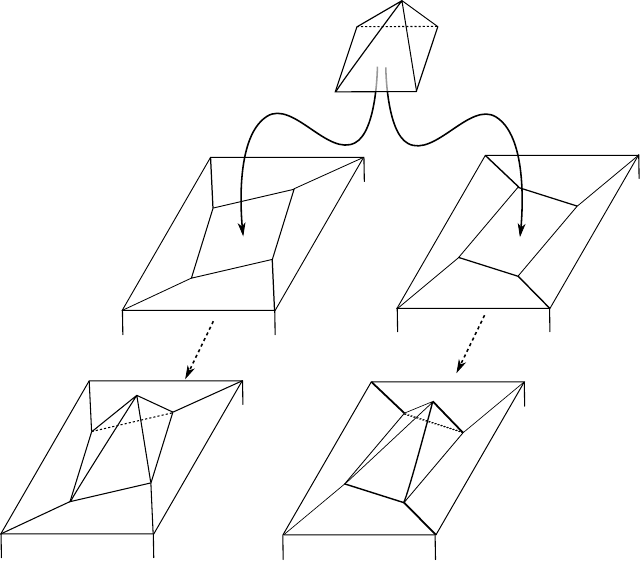}
			\caption{A counterexample that allows  flat edges but has no flat vertices  }
			\label{fig:nonconvex_pi_angle}
		\end{center}
	\end{figure}	
\end{exam}

Finally, even with the condition that all faces are strictly-convex and there are no flat edges, a counterexample still exists, as illustrated below.

\begin{exam}[partially-flat vertices]\label{ex:partially-flat-vertex}
Consider two regular cubes, each with six quadrilateral faces. On each of these faces, draw one of two graphs consisting of eight isosceles triangles and one square, as shown in Figure~\ref{fig:partiallyflat}.
Similar to Example~\ref{flatedge}, we can consider two different choices of twisting direction: left and right. In each case, the corresponding shaded isosceles triangles or square regions are congruent. Using this construction, we can create two different drawings on the cubes such that it is impossible for the two cubes to be congruent to each other by choosing various twisting directions on each face.
	\begin{figure}[H]
		\begin{center}
			\includegraphics{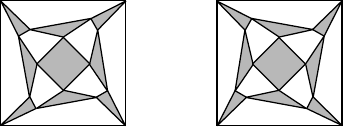}
			\caption{Corresponding faces for the counterexample  with strictly-convex faces and without flat edges}
		\label{fig:partiallyflat}
		\end{center}
	\end{figure}	
	\vspace{-1em}
As the final step, we attach a tetrahedron (or a 4-pyramid) to each shaded face. This is achieved by gluing the region to the base of the tetrahedron (or 4-pyramid), ensuring the corresponding edges match in length and dihedral angle. This process results in figures with strictly convex faces and no flat edges. However, they cannot be congruent to each other.
\end{exam}

%\begin{defn}
	
%\end{defn}

To summarize the discussion so far, in order to establish uniqueness  by dihedral angles and edge lengths,  at least two conditions are required: all faces are convex, and no vertices are partially-flat.

\section{Generalization to nonconvex polyhedra}\label{generalizedNonCon}
A basic idea for extending to the nonconvex case is essentially similar to Theorem~\ref{convexRigidity} for the strictly-convex polyhedra. Even with a nonconvex polyhedron, we are trying to find a vertex $v$ such that   $\N_v$ is uniquely determined by their dihedral angles and edge lengths and performing  vertex reduction at $v$.
%, and we can successively reduce them one by one.

\subsection{Strongly-rigid vertex}
Let us begin by considering a combinatorially stronger concept than  rigid vertex.  We recall that $\tau(v)$ represents the number of non-triangular faces adjacent to a vertex $v$.
\begin{defn}\label{def:strongly-rigid}
	A vertex $v$ is \emph{strongly-rigid} if %one of the followings holds,
	\begin{enumerate}[\hspace{13em}(i)] 
		\item $\tau(v) \leq 3 $   for $ \deg(v) \leq  4$,	
		\item $\tau(v) \leq 1  $  \text{ for } $\deg(v)\geq 5$.
%		\item $\tau(v) = 0  $  \text{ for } $\deg(v)\geq 6$.
	\end{enumerate}
\end{defn}
The condition for a vertex to be strongly-rigid is a stronger requirement compared to being rigid, but the proof of existence follows a nearly identical approach. Similar to Lemma~\ref{lem:existencerigid}, we establish the existence of a strongly-rigid vertex, as follows.

\begin{lem}\label{exist-strong}
	For every polyhedral graph, 	there always exists a strongly-rigid vertex.
\end{lem}
\begin{proof}
	Suppose that a polyhedral graph $P$ has no strongly-rigid vertex. Let us count the adjacent triangular faces at each vertex. There is no 3-valent vertex since all 3-valent vertices are strongly-rigid. If a 4-valent vertex has any triangular face, it becomes strongly-rigid. 
So  there is no room for triangular faces at 3-valent or 4-valent vertices.
Consequently,  the number of triangular faces is at most $(n-2)$ for each $n$-valent vertex with $n\geq 5$.  We  express this inequality as:
	$$
	3F_3 \leq 3V_5 +4V_6+5V_7+\cdots.
	$$
	At the same time, we already have the following inequality from Lemma~\ref{lem:v3f3}. 
	$$3 F_3\geq 3 (V_5+2V_6+3V_7+\cdots+8 )= 3V_5+6V_6 + 9V_7+\cdots +24.$$
	These two inequalities contradictory to each other and it completes the proof. 
\end{proof}

\subsection{Vertex reduction}

When dealing with strictly-convex polyhedra, taking the convex hull is an appropriate method for performing vertex reduction. However, in nonconvex cases, the convex hull approach is not suitable. To address this, we first define vertex reduction combinatorially as follows.
 
\begin{defn}[vertex reduction] \label{def:vertexreduction}
Let  $P$ be a polyhedral graph $P$ and $P_v$ be the \emph{local triangulation} of $P$ at a vertex  $v$.  
Let $P'$ be obtained from $P_v$ by replacing the neighborhood~$\N_v$ of $v$ with $\tN_v$ by identifying the boundary $\bdN_v$:
$$P':=(P_v \ssetminus v) \bigcup\limits_{\bdN_v}  \tN_v \,,$$
where $P_v\ssetminus v$ is the subgraph obtained from $P_v$ by removing the vertex $v$ and all edges adjacent to $v$, and $\tN_v$ is a triangulation of  $\bdN_v$ without adding any additional vertices.	
We call $P'$ the \emph{vertex reduction} of $P$ at a vertex $v$.
\end{defn}
 Note that it is crucial to ensure that newly-added faces in $P'$ are always triangles. This is very important for handling flat edges in Proposition~\ref{rigidityvertexreduction}.
It is obvious that  $\V(P') = \V(P) \ssetminus v$ and 
the number of vertices  of $P'$ is one less than $|\V(P)|$.
The following proposition establish that we can always find a sequence of vertex reductions that consists of polyhedral graphs.

\begin{prop}[existence of a reduction sequence] \label{reductionSeq}
	Let us consider a  polyhedral graph~$P$.
	Then, there is a sequence of polyhedral graphs $P_0(=P)$, $P_1$, \dots, $P_N$ such that $P_{i+1}$ is the vertex reduction of $P_i$ at a strongly-rigid vertex  and  $P_N$ has a strongly-rigid vertex $v$ with $\V(\N_v)=\V(P_N)$.
\end{prop}
\begin{proof}
	We only need to demonstrate that a single vertex reduction is possible. 
	To better understand each step, refer to Figure~\ref{fig:vertexreductionexample}. 
	Firstly, by Lemma~\ref{exist-strong}, we can always find a strongly-rigid vertex $v$ in the polyhedral graph $P (\neq P_N)$. 
We  assume that	$\V(\N_v)\subsetneq \V(P_N)$, otherwise it signifies   termination condition of the induction.
     Then, we  replace $\N_v$ with $\tN_v$ along with $\bdN_v$.
	 Note that if you make an arbitrary choice of $\tN_v$ during vertex reduction, the resulting graph $P'$ may not remain 3-connected as shown in Figure~\ref{fig:vertexreductionexample}~(c).
	 	
	 \begin{figure}[H]
	 	\begin{center}
	 		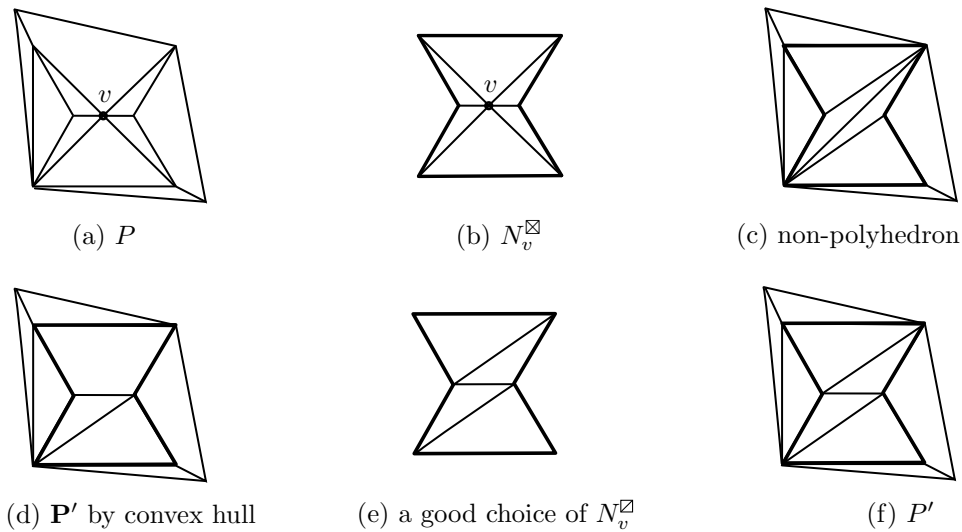
	 		\caption{ $P'$ is the vertex reduction at $v$ of $P$. }
	 		\label{fig:vertexreductionexample}
	 	\end{center}
	 \end{figure}
	 	 
	To ensure that $P'$ is a polyhedral graph, we first consider another arbitrary strictly-convex realization $\mathbf{P}$ of  $P$ by Steinitz's theorem. After removing $v$ of $\mathbf{P}$, we obtain a reduced realized polyhedron $\mathbf P'$ by performing  convex hull, as in Theorem~\ref{convexRigidity}. 
It is clear that $\V(\mathbf P)\ssetminus v$ cannot lie in a plane. If it were otherwise, given that $\mathbf P$ is a strictly-convex realization, we would have $\V(\N_v) = \V(P)$. Consequently, the convex hull $\mathbf P'$ must be a 3-dimensional strictly-convex polyhedron, and its combinatorial type must correspond to a polyhedral graph.
	If a non-triangular face appears in $\mathbf P'$ as illustrated in Figure~\ref{fig:vertexreductionexample}~(d), we subdivide it into triangles further without introducing a new vertex. 
By this strategy, we can always find  a good choice of local triangulation $\tN_v$ as in Figure~\ref{fig:vertexreductionexample}~(e) and 
 obtain a polyhedral graph~$P'$ from the combinatorial type of $\mathbf P'$ as in Figure~\ref{fig:vertexreductionexample}~(f). Note that the combinatorial type of $\mathbf P'$ is not the same as $P'$ itself, as $P'$ may have additional subdivisions.
	This induction procedure always concludes when $\V(\N_v)= \V(P)$, as described in Theorem~\ref{convexRigidity}. We can thus obtain a finite sequence  $P_0(=P)$, $P_1$, \dots, $P_N$ of polyhedral graphs, progressively reducing the number of vertices.
\end{proof}
When we perform vertex reduction on $P$ to obtain $P'$, the newly modified faces and edges come from $\tN_v$. These are referred to as \emph{newly-added faces} and \emph{newly-added edges}, respectively. The fact that all newly-added faces are triangles will play a crucial role in Section \ref{sec:proofmain}.

Remark that in this paper, there is often no need to distinguish between a polyhedral graph and its realization. However, in this proof, such a distinction is crucial. For this reason, we introduce specific notations: the usual $P$ to denote a combinatorial graph, and the bold $\mathbf P$ to denote  a 3-dimensional geometric polyhedron.

\subsection{Proof of the main theorem}\label{sec:proofmain}
The strategy to prove the main theorem (Theorem~\ref{mainThm}) is quite simple, as  sketched  below.
For a given geometric polyhedron $P$, we consider a reduction sequence $P_0=P$, $P_1$, $\dots$, $P_N$, where each $P_i$ has a corresponding strongly-rigid vertex $v_i$, according to Proposition \ref{reductionSeq}.
 When  the induction terminates at the case of $\V(\N_v)=\V(P)$, the geometric realization of $\N_v$  determines the final $P_N$ uniquely up to congruence. This is because  two polyhedra with the same combinatorial type must be congruent if their vertex sets  are realized up to isometry.
Therefore, the important task is to prove that $\N_{v_i}$ is uniquely determined by the dihedral angles and edge lengths of $P_i$. Then, the dihedral angles and edge lengths of $P_{i+1}$  can be computed from those of $P_i$ along with the vertex positions of $\N_{v_i}$, vice versa. 

To geometrically obtain $\N_v$, we need Condition (i) which requires convex faces. This is necessary because we must be able to cut geometrically along $\bdN_v$ for non-triangular faces adjacent to $v$. Then, we can proceed with the vertex reduction: first by cutting off $\N_v$ from $P_v$, and then by gluing $\tN_v$ along $\bdN_v$. This process is not merely combinatorial; it is geometrically well-behaved due to the rigidity of $\N_v$. 

Furthermore, throughout this iterative process, we need Condition (iii) to ensure that newly-added faces are never degenerate.  Facial angles of $0$ or $\pi$ indicate collinearity of the triangle's vertices.
Note that, even in spherical polyhedra, we can choose the realization of $\tN_v$ such that it has only convex triangular faces. This is possible by selecting edges with lengths less than~$\pi$ since there is no antipodal pair of vertices in $P$ due to Condition (iii) in Theorem~\ref{mainThm}.
The induction continues until we arrive at $P_N$, at which point the induction terminates. Then we can reconstruct the original polyhedron $P$ uniquely along the reverse sequence $P_N, P_{N-1}, \dots, P_0(=P)$.

Consequently, the remaining task is to prove the rigidity of $\N_v$ for any strongly-rigid vertex~$v$ that emerges during reduction  procedures.
Note that a flat edge may be generated during vertex reduction even though the original polyhedron does not contain any flat edges, as illustrated in Figure~\ref{fig:flatEdge}.

\begin{figure}[H]
	\begin{center}
		%% Creator: Inkscape 1.2.2 (b0a84865, 2022-12-01), www.inkscape.org
%% Accompanies image file 'flatEdge.pdf' (pdf, eps, ps)
\begingroup%
  \makeatletter%
  \providecommand\color[2][]{%
    \errmessage{(Inkscape) Color is used for the text in Inkscape, but the package 'color.sty' is not loaded}%
    \renewcommand\color[2][]{}%
  }%
  \providecommand\transparent[1]{%
    \errmessage{(Inkscape) Transparency is used (non-zero) for the text in Inkscape, but the package 'transparent.sty' is not loaded}%
    \renewcommand\transparent[1]{}%
  }%
  \providecommand\rotatebox[2]{#2}%
  \newcommand*\fsize{\dimexpr\f@size pt\relax}%
  \newcommand*\lineheight[1]{\fontsize{\fsize}{#1\fsize}\selectfont}%
  \ifx\svgwidth\undefined%
    \setlength{\unitlength}{168.8167285bp}%
    \ifx\svgscale\undefined%
      \relax%
    \else%
      \setlength{\unitlength}{\unitlength * \real{\svgscale}}%
    \fi%
  \else%
    \setlength{\unitlength}{\svgwidth}%
  \fi%
  \global\let\svgwidth\undefined%
  \global\let\svgscale\undefined%
  \makeatother%
  \begin{picture}(1,0.43204737)%
    \lineheight{1}%
    \setlength\tabcolsep{0pt}%
    \put(0,0){\includegraphics[width=\unitlength,page=1]{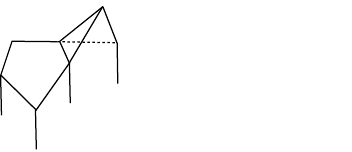}}%
    \put(0.11813724,0.34697341){\color[rgb]{0,0,0}\makebox(0,0)[lt]{\lineheight{1.25}\smash{\begin{tabular}[t]{l}$P$\end{tabular}}}}%
    \put(0.1394004,0.26283056){\color[rgb]{0,0,0}\makebox(0,0)[lt]{\lineheight{1.25}\smash{\begin{tabular}[t]{l}$e$\end{tabular}}}}%
    \put(0.30148991,0.41892613){\color[rgb]{0,0,0}\makebox(0,0)[lt]{\lineheight{1.25}\smash{\begin{tabular}[t]{l}$v$\end{tabular}}}}%
    \put(0,0){\includegraphics[width=\unitlength,page=2]{flatEdge.pdf}}%
    \put(0.78139604,0.33962761){\color[rgb]{0,0,0}\makebox(0,0)[lt]{\lineheight{1.25}\smash{\begin{tabular}[t]{l}$P'$\end{tabular}}}}%
    \put(0.80265921,0.25548477){\color[rgb]{0,0,0}\makebox(0,0)[lt]{\lineheight{1.25}\smash{\begin{tabular}[t]{l}$e$\end{tabular}}}}%
    \put(0.96474873,0.41158038){\color[rgb]{0,0,0}\makebox(0,0)[lt]{\lineheight{1.25}\smash{\begin{tabular}[t]{l}$v$\end{tabular}}}}%
    \put(0,0){\includegraphics[width=\unitlength,page=3]{flatEdge.pdf}}%
  \end{picture}%
\endgroup%

		\caption{ A flat edge $e$ may appear during vertex reduction at $v$.}
		\label{fig:flatEdge}
	\end{center}
\end{figure}

Therefore, it is necessary to consider a vertex reduction involving flat edges.
Fortunately, the following proposition states that higher-degree cases are resolved easily.
\begin{prop}\label{deg5SvRigid}
	Let   $v$ be a strongly-rigid vertex. 
		If $\deg(v) \geq 5$ and each triangle face adjacent to $v$ is convex, then $\N_v$ is uniquely determined by the dihedral angles and edge lengths adjacent to $v$.	
	In particular, 
	even if there are $\pi$ or $0$ dihedral angles near $v$, it does not matter.
\end{prop}
\begin{proof}
	Since every adjacent triangle face is convex,  its interior angle is determined by edge lengths. Therefore, 
	the number of non-triangular faces, and thus the number of unknown facial angles, is at most one. In the vertex figure $S_v$, we can determine all interior angles and edge lengths except for at most one edge. Every polygon can be uniquely constructed by sequentially drawing its edges and connecting them at the interior angles. The polygon is determined without the final edge length since all vertices are determined without the edge.
\end{proof}

The cases of $\deg(v) = 3$ and $\deg(v) = 4$ require special attention  since they  have more than one unknown facial angle. To handle these cases properly we employ `generalized trigonometry', which deals with  angles greater than $\pi$. These will be fully discussed in the subsequent sections, and for now we use the results to obtain the following proposition.

\begin{prop}\label{deg34SvRigid}
	Let $v$ be a strongly-rigid vertex, and let the dihedral angles and edge lengths of $\N_v$ be given. If $\deg(v) = 3$ or $\deg(v) = 4$, then $\N_v$ is uniquely determined under the\ conditions:
	(i) $v$ has no flat edge,
	(ii) all adjacent facial angles are less than $\pi$.
\end{prop}
 \begin{proof}
 In  3-valent cases, the vertex figure $S_v$ is a non-singular triangle with all interior angles given. According to Corollary~\ref{trianglerigidity}, there are two possible realizations of such a triangle, where the edge lengths are either $l$ or $2\pi-l$ (or either $l$ or $l+\pi$). However, since all facial angles are less than $\pi$, we can determine unique realization by fixing one choice.
 The 4-valent case is directly addressed by Theorem~\ref{thm:sphquadRigid}.
 \end{proof}
 
Since flat edges are not in the original $P$, we have the following simple observation: For any flat edge of every polyhedron during vertex reduction, one of the two adjacent faces must be a triangle, because it is  newly-added. Equivalently, for each singular angle in  spherical figures, one of the two adjacent edges must have a known length.
This is very simple but crucial point to be often used. So we  refer to this observation as  \emph{flat edge principle}.
 
Finally, by incorporating the previous propositions and flat edge principle, we obtain the following proposition, which completes the proof of the main theorem.
 \begin{prop}\label{rigidityvertexreduction}
 Let $P$ satisfy the three conditions of the main theorem and all dihedral angles and edge lengths be given. 
 If $v$ be a strongly-rigid vertex of the polyhedron $P_i$ in the reduction sequence, then $\N_v$ is uniquely determined. 
 \end{prop}
\begin{proof}
	The remaining case not covered by Proposition~\ref{deg34SvRigid} and Proposition~\ref{deg5SvRigid} is the situation where a 3- or 4-valent vertex has a flat edge.
First,	let us consider 3-valent case.
By the classification of singular triangle of Proposition~\ref{singular3-gon}, the only possibilities are $(A,B,C)=(\pi,\pi,\pi)$, $(0,0,\pi)$, and $(\theta,\theta,\pi)$. 
For the cases where $(A,B,C) = (\pi, \pi, \pi)$ or $(0,0,\pi)$, according to flat edge principle, at least two edge lengths among the three edge lengths of $S_v$ are determined, and consequently, $\N_v$ is also determined.
 For $(A,B,C)=(\theta,\theta,\pi)$, 
 by flat edge principle, at least one edge length adjacent to the singular angle $C$ must be determined  and consequently, $S_v$ is also determined. 
 In any case, any strongly-rigid 3-valent vertex during  vertex reduction must carry sufficient information to determine $S_v$.

Secondly, let us consider 4-valent cases. 
 By  flat edge principle similar to the 3-valent cases, at least one face adjacent to the flat edge must be a triangle, and the corresponding edge length of $S_v$ is also determined. If only one flat edge occurs, it is reduced to the triangle case without singular angles. 
 If exactly two flat edges occur, it is also reduced to the bigon case. These cases are handled by the same way as 3-valent cases. 
 The final case is when all edges are flat. 
In this case, by flat edge principle, the pattern of four adjacent faces must be either (triangle, triangle, triangle, non-triangle) or (triangle, non-triangle, triangle, non-triangle), as non-triangular faces cannot be consecutive.
The former case is easily resolved as only one facial angle remains unknown. The latter case presents the worst possible scenario. However, this scenario is precluded by Condition (ii), which prohibits any partially-flat  vertex. If all faces adjacent to the vertex  lie on the same plane, at most one of these faces can be a face of the original $P$. The worst scenario, however, would require two non-triangular faces which cannot be newly-added. Therefore, this worst-case scenario does not occur.
\end{proof}  
Note that the requirement of the partially-flat vertex condition is not necessary except in the final worst-case scenario. In fact, the occurrence of this situation can be avoided by carefully selecting a suitable reduction sequence in many practical examples even containing partially-flat vertices. It is, however,  important to note that the partially-flat vertex condition cannot be replaced by the flat-edge condition alone.
This is evident from the counterexample provided in Example~\ref{ex:partially-flat-vertex}, which demonstrates that even without flat edges, uniqueness can fail. It is worth noting that although the counterexample exists, it does not form a continuous family. Therefore, there is still a possibility that local rigidity holds true, as addressed in Question~\ref{ques}.

As we consider the proof of Theorem~\ref{mainThm}, it becomes apparent that if we prevent the presence of a newly-added 4-valent flat edge in the alternating pattern (triangle, non-triangle, triangle, non-triangle), we can establish the uniqueness result. As a straightforward application of this strategy, we obtain the proof of Theorem~\ref{thm2} as follows. 
\begin{proof}[Proof  of Theorem~\ref{thm2}] 
To create a 4-valent flat vertex in the alternating pattern during  vertex reduction, at least seven vertices of $P$ lying on the same plane is required in Figure~\ref{fig:7vertexthm}.
\end{proof}

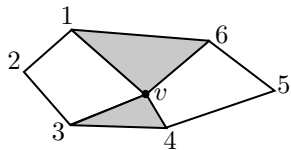
\begin{figure}[H]
	\begin{center}	
		%% Creator: Inkscape 1.2.2 (732a01da63, 2022-12-09), www.inkscape.org
%% Accompanies image file '7vertexthm.pdf' (pdf, eps, ps)
\begingroup%
  \makeatletter%
  \providecommand\color[2][]{%
    \errmessage{(Inkscape) Color is used for the text in Inkscape, but the package 'color.sty' is not loaded}%
    \renewcommand\color[2][]{}%
  }%
  \providecommand\transparent[1]{%
    \errmessage{(Inkscape) Transparency is used (non-zero) for the text in Inkscape, but the package 'transparent.sty' is not loaded}%
    \renewcommand\transparent[1]{}%
  }%
  \providecommand\rotatebox[2]{#2}%
  \newcommand*\fsize{\dimexpr\f@size pt\relax}%
  \newcommand*\lineheight[1]{\fontsize{\fsize}{#1\fsize}\selectfont}%
  \ifx\svgwidth\undefined%
    \setlength{\unitlength}{105.47010998bp}%
    \ifx\svgscale\undefined%
      \relax%
    \else%
      \setlength{\unitlength}{\unitlength * \real{\svgscale}}%
    \fi%
  \else%
    \setlength{\unitlength}{\svgwidth}%
  \fi%
  \global\let\svgwidth\undefined%
  \global\let\svgscale\undefined%
  \makeatother%
  \begin{picture}(1,0.47409265)%
    \lineheight{1}%
    \setlength\tabcolsep{0pt}%
    \put(0,0){\includegraphics[width=\unitlength,page=1]{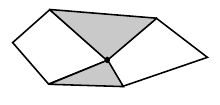}}%
    \put(0.51661791,0.17937906){\color[rgb]{0,0,0}\makebox(0,0)[lt]{\lineheight{1.25}\smash{\begin{tabular}[t]{l}$v$\end{tabular}}}}%
    \put(0.186904,0.45199372){\color[rgb]{0,0,0}\makebox(0,0)[lt]{\lineheight{1.25}\smash{\begin{tabular}[t]{l}$1$\end{tabular}}}}%
    \put(-0.00079848,0.28985407){\color[rgb]{0,0,0}\makebox(0,0)[lt]{\lineheight{1.25}\smash{\begin{tabular}[t]{l}$2$\end{tabular}}}}%
    \put(0.15590306,0.03852842){\color[rgb]{0,0,0}\makebox(0,0)[lt]{\lineheight{1.25}\smash{\begin{tabular}[t]{l}$3$\end{tabular}}}}%
    \put(0.55071276,0.00589972){\color[rgb]{0,0,0}\makebox(0,0)[lt]{\lineheight{1.25}\smash{\begin{tabular}[t]{l}$4$\end{tabular}}}}%
    \put(0.95313584,0.20401406){\color[rgb]{0,0,0}\makebox(0,0)[lt]{\lineheight{1.25}\smash{\begin{tabular}[t]{l}$5$\end{tabular}}}}%
    \put(0.7345222,0.38230265){\color[rgb]{0,0,0}\makebox(0,0)[lt]{\lineheight{1.25}\smash{\begin{tabular}[t]{l}$6$\end{tabular}}}}%
  \end{picture}%
\endgroup%

	\end{center}
	\caption{
	To obtain a flat vertex $v$ with an alternating pattern during vertex reductio, at least six vertices in $\bdN_v$ must lie in the same plane as the vertex $v$. The shaded regions represent the newly-added faces.}
		\label{fig:7vertexthm}
\end{figure}

\section{Rigidity for spherical polygons}\label{nonconSph}
In order to study the rigidity of spherical polygons induced from the vertex figures of nonconvex polyhedra, we need to formulate a generalized form of trigonometry. This should account for triangles that may self-intersect or have interior angles or edge lengths exceeding $\pi$.
We refer to this as \emph{generalized trigonometry} and  provide a proof for the convenience of the reader. Although there is a rather older reference \cite{klein_ubersicht_1933}, most modern texts and references only prove this under the assumption of strictly convex triangles ($0<\theta<\pi$). While more recent sources covering this generalized case may exist, we have been unable to find them.

\subsection{Generalized trigonometry
} \label{sphTrigonometry} 

Let us consider a spherical triangle with three edge lengths $a, b, c$ and the opposite interior angles $A$, $B$, $C$, respectively. The classical spherical trigonometry formulas are as follows, and they are well-known under the assumption that both edge lengths and angles are between 0 and $\pi$.
\begin{align}
	\sin a \sin B &=\sin A \sin b,& \text{(Sine rule)} \nonumber\\
	\sin b \sin C &=\sin B \sin c,&  \nonumber\\
	\sin c \sin A &=\sin C \sin a & \nonumber \\
	\cos a &= \cos b \cos c + \sin b \sin c \cos A   &\text{(Cosine rule)}\nonumber\\
	%	 \label{eq:trigonometry}\\
	\cos b &= \cos c \cos a + \sin c \sin a \cos B   &\nonumber\\
	\cos c &= \cos a \cos b + \sin a \sin b \cos C   &\nonumber\\
	\cos A &= -\cos B \cos C + \sin B \sin C \cos a 
	&\text{(dual Cosine rule)} \nonumber\\
	\cos B &= -\cos C \cos A + \sin C \sin A \cos b &\nonumber\\
	\cos C &= -\cos A \cos B + \sin A \sin B \cos c &\nonumber
\end{align}

\begin{prop}\label{sinecosine}  
	Spherical trigonometry, which consists of the Sine rules, Cosine rules, and dual Cosine rules, holds true even for spherical triangles that may self-intersect or have interior angles or edge lengths within the range of $[0,2\pi]$. 
\end{prop}

\begin{proof}
	
	Firstly, let us consider three points in non-collinear position in $S^2$. Recall that there are exactly two choices for the geodesic segment connecting two given points: one with  length less than $\pi$ and the other with  length greater than $\pi$.
	If they have the same length, the two points are antipodal and collinear together with   the third point, we will deal with this collinear situation at the end of  the proof.
	
	Therefore,	when forming a triangle with edge lengths less than $\pi$ without self-intersection, there are only two choices of interior angles as in Figure \ref{fig:smalltriangle}.
Let the tuple $(a,b,c;A,B,C)$ denote the lengths of the three edges $(a, b, c)$ and the three interior angles $(A, B, C)$ of a triangle. Both cases share the same edge lengths $a$, $b$, $c$, but the interior angles change from $A$, $B$, $C$ to $2\pi-A$, $2\pi-B$, $2\pi-C$ in the second case. We can easily verify that the trigonometric formulas for the $(a,b,c;2\pi-A,2\pi-B,2\pi-C)$ case are deduced from the strictly-convex case of $(a,b,c;A,B,C)$.
		\begin{figure}[h!]
		{
			\begin{align*}
				\vcenter{\hbox{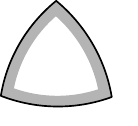} } ~~~~&&~ 
				\vcenter{\hbox{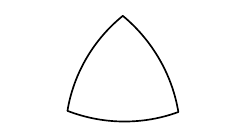} } 
			\end{align*}
			\caption{Two triangles with the same edge lengths less than $\pi$, but with different corresponding interior angles}
			\label{fig:smalltriangle}
		}
	\end{figure}
	
Secondly, we proceed to examine cases where triangles are not strictly-convex. 
Given that there are two geodesic segments connecting the points for each pair of points, we have  $2\times 2\times 2=8$ possible configurations  of triangles based on selection of the longer and shorter geodesic segments.
	Accounting for permutations in selection of the `longer' or `shorter' geodesics, we identify four essentially distinct triangles, as described in Figure \ref{fig:largeangledcases}.
	\begin{figure}[h!]
		{
			\begin{align*}
				\vcenter{\hbox{\includegraphics[scale=0.7]{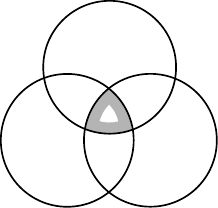}} } ~~~~&&~ 
				\vcenter{\hbox{\includegraphics[scale=0.7]{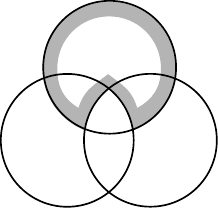}} } \\
				1 \times \text{(short,short,short)}  &&
				3 \times \text{(long,short,short)} \\
				\vcenter{\hbox{\includegraphics[scale=0.7]{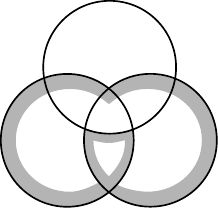}} } ~~~&&~~~
				\vcenter{\hbox{\includegraphics[scale=0.7]{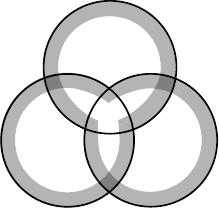}} } \\
				3 \times \text{(long,long,short)} &&
				1 \times \text{(long,long,long)}
			\end{align*}
		}	
		\caption{There are essentially four types of  triangles.}\label{fig:largeangledcases}
	\end{figure}
	We already know spherical trigonometry holds for the strictly-convex triangle (short, short, short).
	Then, as shown in Figure~\ref{fig:largeangledcases}, we can see the other three cases as follows: 
	
	$$\aligned
	(a,b,c;A,B,C) &\rightarrow(2\pi-a,b,c;2\pi-A,\pi-B,\pi-C), \\
	&\rightarrow(a,2\pi-b,2\pi-c;2\pi-A,\pi-B,\pi-C),\\
	&\rightarrow(2\pi-a,2\pi-b,2\pi-c;2\pi-A,2\pi-B,2\pi-C).
	\endaligned$$
	Based on this rule change, we can deduce the trigonometry formulas for all possible configurations of non-singular triangles from the usual trigonometry formulas.
	
Even for the case where the three vertices are in collinear position, the formulas of generalized trigonometry are also satisfied. This is easily established by applying the limit and continuity argument from the case of general position.
\end{proof}

\subsection{Rigidity for spherical triangles}\label{sec:trianglerigidity}
We only consider  spherical triangles with $$ 0 < a, b, c < 2\pi \text{ and } 0\leq  A, B, C \leq 2\pi.$$  
Let a triangle be \emph{singular} if there is a singular interior angle. Note that if an edge length is $\pi$ then  the triangle must be singular by the Sine rule. 
If we consider spherical triangles with  $0<a, b, c, A, B, C <\pi$, then 
$SSS$, $ASA$, $SAS$ and $AAA$ congruence conditions guarantee  uniqueness, just as they do for hyperbolic  or Euclidean triangles (excluding the $AAA$ condition for Euclidean triangles).
However, if one allows interior angles or edge lengths greater than $\pi$, the uniqueness ceases to hold.
Let us begin with the following  existence and uniqueness result.

\begin{prop}\label{triexistence}
	Consider  a  tuple $(a,b,c;A,B,C)$ of real numbers in $(0,\pi)\cup(\pi,2\pi)$ satisfying the generalized trigonometry as described in Proposition~\ref{sinecosine}. Then, there is a unique triangle that realizes these lengths and angles.
\end{prop}
\begin{proof}
Construct an angle $A$ and extend the adjacent edges $b$ and $c$ with given lengths. It is important to note that this construction is possible without any constraints such as the triangle inequality. On the sphere, there are two choices of edges $x$ and $y$ that connect the endpoints of $b$ and $c$ along the great circle. Due to the non-singular condition, the lengths of $x$ and $y$ are not the same. This allows us to choose one of them as the edge $a$. Thus, a unique triangle is formed. All the trigonometric formulas must be satisfied, and the given values of $a$, $B$, $C$ must match those  depicted in the drawing.
\end{proof}
Thus, there exists an exact correspondence between algebraic solutions to the equations of trigonometry and (potentially nonconvex or self-intersecting) spherical triangles. This relationship yields the following corollary:

\begin{cor}\label{trianglerigidity}
Consider non-singular spherical triangles.
Given $SSS$, $SAS$, $ASA$, and $AAA$ conditions, exactly two distinct triangles exist up to congruence. In particular, if one realization has an edge length or an interior angle of $\theta$, then the corresponding edge length or interior angle of the other realization is either $2\pi-\theta$ or $\theta+\pi$.
\end{cor}
\begin{proof}
Suppose that three values of either $SSS$, $SAS$, $ASA$, and $AAA$ are given. There are three remaining angles or lengths not given. The cosine values $\cos \theta$ of these three values are consecutively determined by the Cosine rule or  dual Cosine rule, so they must be either $\theta$ or $2\pi-\theta$. After choosing any one of these values, the remaining choices are uniquely determined as $2\pi-\theta$ or $\pi+\theta$ by the Sine rule. The pair of tuples obtained in this way is the exact solution to the system of equations from the trigonometry. Due to Proposition~\ref{triexistence} along with the non-singular condition, exactly two distinct triangles are obtained. In particular, the pair of the corresponding edge length and interior angle must be either $(\theta, 2\pi-\theta)$ or $(\theta, \pi+\theta)$.
\end{proof}
For example, in $ASA$ condition, $A$, $b$, and $C$ are given. The dual Cosine rule determines $\cos B$, and $\cos a$ and $\cos c$ are determined by the dual Cosine rule again. If one chooses either $B$ or $2\pi - B$, then the other values of $a$ and $c$ are determined from the two options, respectively, using the Sine Rule. 
Finally, we have two distinct tuples $(a, b, c; A, B, C)$ and $(\pi + a, b, \pi + c; A, 2\pi - B, C)$ which are the solution to  generalized trigonometric formula. These are  exactly two distinct triangles by Corollary~\ref{trianglerigidity}.

For singular triangles, we have the following classification. 

\begin{prop}\label{singular3-gon}
If a spherical triangle has an angle or an edge length equal to  $0$, $2\pi$ or $\pi$, then the only possible angle combinations are $(\pi, \pi, \pi)$, $(0, 0, \pi)$, $(2\pi,2\pi,\pi)$ and $(\theta, \theta, \pi)$. 
\end{prop}	
\begin{proof}
	By applying the generalized trigonometry of Proposition~\ref{sinecosine}, we  straightforwardly check all the possibilities. Note that we don't allow the lengths to be zero by definition. 
\end{proof}

\subsection{Nonrigid spherical polygons}

Let's consider examples in which the nonconvex spherical polygons, induced from the vertex figure $S_v$, are not uniquely determined by the constraints that guaranteed the rigidity of strictly-convex polygons
For example,  Theorem~\ref{rigiditySphPolygon} is no longer valid for nonconvex polygons. The following  is an example of a non-rigid spherical  hexagon with all interior angles given and $n-2$ edge lengths provided. 
Counterexamples of this kind can be easily constructed for an $n$-gon ($n \geq 5$) using similar methods.
\begin{exam}\label{ex:deg6flex}
	Consider   two spherical polygons of six edges as depicted in  Figure~\ref{fig:vertexshape_counterexample}. Each pair of corresponding interior angles in these polygons is equal, i.e., $\theta_i=\theta'_i$ for $1\leq i\leq 6$. Four pairs of corresponding edge lengths are equal, i.e., $l_i=l'_i$ for $i=2,3,5,6$, but the remaining two edge lengths are different, i.e., $l_i\neq l'_i$ for $i=1,4$.  These polygons can be constructed by attaching a smaller spherical triangle above and a larger one below, both positioned differently along the same edge. Note that  this example appears in the vertex figure at a partially-flat vertex since the edges corresponding to $l_1$ and $l_4$ are collinear.

\end{exam}	
\begin{figure}[H]
	\begin{center}
		\scriptsize
		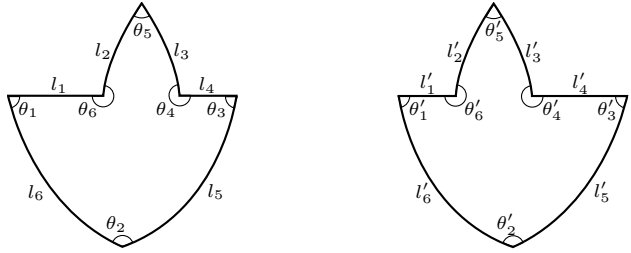
		\caption{Two spherical polygons differing only in two corresponding edge lengths}
		\label{fig:vertexshape_counterexample}
	\end{center}
\end{figure}

We have already established  that every spherical triangle with an edge length of $\pi$ cannot be determined by its interior angle, in Proposition \ref{singular3-gon}. A similar problem also arises in spherical $n$-gons for $n>3$. We present an example of a flexible spherical quadrilateral where all interior angles are specified, but the length of an edge is given by $\pi$.
Note that, in the vertex figure of a polyhedron with strictly-convex faces, this situation does not occur.
\begin{exam}\label{sph-quad}
	Consider   the spherical polygon of four edges in the below Figure~\ref{fig:SphquadPi}. There are  
	four interior angles $\frac{\pi}{2}, \frac{3\pi}{2},\frac{\pi}{2},\frac{3\pi}{2}$ (more generally, $\alpha, \beta+\pi,\beta, \alpha+\pi$ by rotation of the edges of length $\pi$ with fixed four vertices) and four edge lengths $\theta, \pi,\theta,\pi$.  Here we can change the value of $\theta$ freely. The four point $A, B, C$ and $D$ exist on a great circle.
	
\end{exam}
\begin{figure}[H]
	\begin{center}
		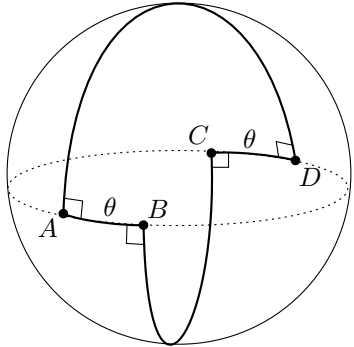	     		
		\caption{A family of non-isomorphic spherical quadrilaterals sharing two edge lengths  and all dihedral angles where the two edge lengths are $\pi$}
		\label{fig:SphquadPi}  	
	\end{center}
\end{figure}

The vertex figure of a strongly-rigid (or rigid) 4-valent vertex is a spherical quadrilateral with specified interior angles and one known edge length.  This information is sufficient to uniquely determine the spherical convex quadrilateral according to Theorem~\ref{rigiditySphPolygon}.
However, in the nonconvex case, there exist nonconvex spherical quadrilaterals that share one edge length and all interior angles, as shown in the following example.

\begin{exam}\label{nonconvexQuad}
	Let us consider two spherical quadrilaterals  on the unit sphere in $\RR ^3$ as in Figure~\ref{fig:nonconvexQuad} where  the Cartesian coordinates can be explicitly given for small positive real numbers $\theta$ and $\phi$.
	\begin{figure}[H]
		\begin{align*}
			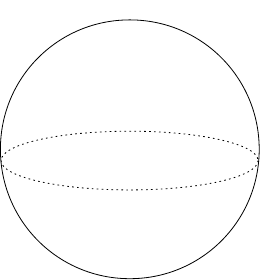 ~~~~~~~&~~~~~~~~
			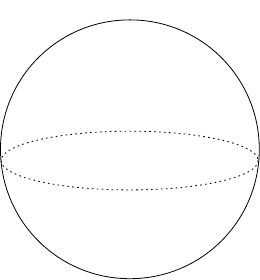 \\	
			\left\{ 
			\begin{aligned}
				A &=(0,0,1) \\
				B &= (\cos \theta, -\sin \theta, 0)   \\
				C &= (-1,0,0) \\
				D &= (\cos \phi \cos \theta,\cos \phi \sin \theta, -\sin \phi )   	
			\end{aligned}	
			\right. ~~~&~~
			\left\{ 
			\begin{aligned}
				A' &=(0,0,1) \\
				B' &=(\cos 3\theta, -\sin 3\theta,0)  \\
				C' &= (-1,0,0) \\
				D' &=(\cos \phi \cos \theta,-\cos \phi \sin \theta, \sin \phi  )    	
			\end{aligned}	
			\right.
		\end{align*}
		
		\caption{Two non-isomorphic  nonconvex spherical polygons   sharing one edge length and all interior angles}
		\label{fig:nonconvexQuad}  	
	\end{figure}
	
	We can compute all interior angles and edge lengths by spherical trigonometry.
	Let us take the bounded region of the polygon as the southern part in Figure~\ref{fig:nonconvexQuad} to determine interior angles. 
	All pairs of interior angles are the same, i.e.
	\begin{align*}
		\angle A&=\angle A'=2 \theta, & \angle B &=\angle B'=\frac{3\pi}{2},\\
		\angle C&=\angle C'=\pi-\cos^{-1}(\frac{\cos \phi \sin \theta}{\sqrt{1-\cos^2 \phi \cos^2 \theta }}), 
		& \angle D &=\angle D'=\pi+\cos^{-1}(\frac{\cos \theta \sin \phi }{\sqrt{1-\cos^2 \phi \cos^2 \theta}})
	\end{align*}   
	and a pair of edge lengths are also same, $\edgepath{AB}=\edgepath{A'B'}=\frac{\pi}{2}$, 
	but the other pairs of edge lengths are different, for instances, 
	\begin{align*}
		\edgepath{BC}&=\pi-\theta, &\edgepath{B'C'}&=\pi-3\theta,\\
		\edgepath{CD}&=\pi-\cos^{-1}(\cos \phi \cos \theta), &\edgepath{C'D'}&=\pi+\cos^{-1}(\cos \phi \cos \theta),\\
		\edgepath{DA}&=\frac{\pi}{2}+ \phi,  &\edgepath{D'A'}&=\frac{\pi}{2}- \phi.
	\end{align*}
	
\end{exam}

Note that $\edgepath{C'D'} > \pi$, so we need to exclude this situation to establish uniqueness.
This is another reason to require the convex face condition, which ensures that all edge lengths are less than $\pi$.
In fact, we establish that if the interior angles of a spherical quadrilateral lie within the range $(0,\pi)\cup (\pi,2\pi)$, and the edge lengths lie within $(0,\pi)$, then the quadrilateral is uniquely determined by all its interior angles and  one edge length. These conditions are guaranteed by the requirements of strictly-convex faces and non-flat edges.

\subsection{Rigidity for spherical quadrilaterals}

Finally, we prove a rigidity theorem for spherical quadrilaterals, which is a generalization of Theorem~\ref{rigiditySphPolygon}. The  idea of the proof is quite straightforward: we examine all possible constructions under the given constraints, and then confirm that no simultaneous realizations can occur among them.

\begin{thm}\label{thm:sphquadRigid}
Let  $\polygon{v_1\. v_2\. v_3\. v_4 }$ be a spherical quadrilateral, possibly nonconvex and self-intersecting, that satisfies the following conditions: (i) all edge lengths are less than $\pi$, and (ii) all interior angles are not singular.
Under these conditions, the spherical quadrilateral  is uniquely determined by one edge length and all interior angles.
\end{thm}

\begin{proof}
	Let $\theta_i$ be the interior angle at $v_i$ for $i=1,\dots,4$.
	Without loss of generality, we assume that the length of the edge $\edgepath{v_1\. v_2}$ is given 	
and	there remains essentially three cases as follows, 
	\begin{enumerate}[\hspace{7em} (i)]
		\item $\theta_1  < \pi $ and $\theta_2 < \pi$ ,
		\item $\theta_1< \pi $ and $\theta_2 > \pi$ without self-intersection,  	
		\item $\theta_1< \pi $ and $\theta_2 > \pi$ with self-intersection.
	\end{enumerate}
	
In case (i), we consider drawing a geodesic with interior angle $\theta_i$ to the edge $\edgepath{v_1\. v_2}$ at each $v_i$ for $i=1,2$. Each geodesic forms a great circle on $\mathbb{S}^2$ and intersects each other twice. We denote the first intersecting point as $v_5$. Each great circle is divided into two geodesic segments by $v_1$ (or $v_2$) and $v_5$.
The vertex $v_3$ (or $v_4$) can be placed on one of the two geodesic segments. We denote the position of $v_3$ (or $v_4$) as $y$ and $y'$ (or $x$ and $x'$) as shown in Figure~\ref{fig:quad1}. . There are exactly four possible realizations for the pair of vertices $(v_3, v_4)$: one is $(y, x)$ and the others are $(y,x')$, $(y',x)$, ${{(y',x')}}$.
	
	\begin{figure}[h!]
		\begin{center}
			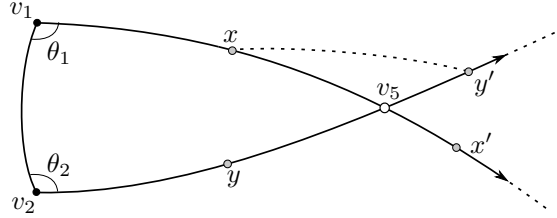
			\caption{The cases of the quadrilateral $\polygon{v_1\. v_2\. v_3\. v_4}$ with  $\theta_1  < \pi $ and $\theta_2 < \pi$} \label{fig:quad1}
		\end{center}
	\end{figure} 
	
As is clear, these four possible realizations of $(v_1 v_2 v_3 v_4)$ cannot occur simultaneously because the configuration $(\theta_3,\theta_4)$ of the interior angles at $v_3$ and $v_4$ are all different: (less than $\pi$, less than $\pi$), (greater than $\pi$, less than $\pi$), (less than $\pi$, greater than $\pi$), and (greater than $\pi$, greater than $\pi$), respectively.
For each of the four cases, the quadrilateral $\polygon{v_1 v_2 v_3 v_4}$ can be considered as the union or subtraction of two triangles, namely $\polygon{v_1\.v_2\.v_5}$ and $\polygon{v_5\.v_3\.v_4}$. For example, if we choose $y'$ and $x$ for $v_3$ and $v_4$, respectively, then $\polygon{v_1\. v_2\. v_3\. v_4} = \polygon{v_1\. v_2\. v_5} \cup \polygon{v_5\. v_3\. v_4}$, as illustrated in Figure~\ref{fig:quad1}.
The triangle $\polygon{v_1\. v_2\. v_5}$ is uniquely determined by $ASA$ condition of Corollary~\ref{trianglerigidity} because all edge lengths are less than $\pi$. Similarly, the triangle $\polygon{v_5\.v_3\.v_4}$ is also uniquely determined by $AAA$ condition because all interior angles are less than $\pi$. \vspace{0.3em}
	
	Let us consider case (ii) where $\theta_1 < \pi$ and $\theta_2 > \pi$, without self-intersection. Based on this constraint, the possible shape of the quadrilateral $\polygon{v_1\. v_2\. v_3\. v_4}$ is as shown in Figure~\ref{fig:NonConQuad}.
	We extend the edge $\edgepath{v_3\. v_2}$ geodesically from $v_2$ into the interior direction of $\polygon{v_1\. v_2\. v_3\. v_4}$. The extended geodesic must intersect the boundary of $\polygon{v_1\. v_2\. v_3\. v_4}$, as illustrated in the figure~\ref{fig:NonConQuad}.
\begin{figure}[h!]
	\begin{center}
		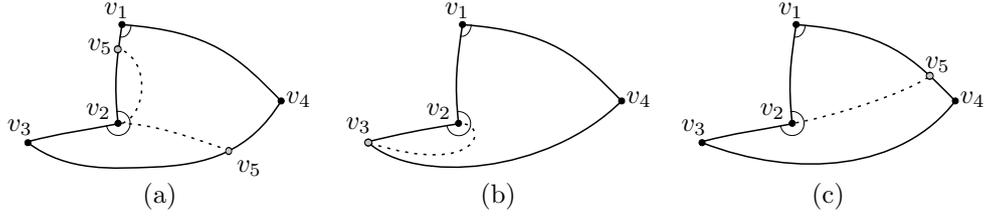
		\caption{The cases of quadrilaterals $\polygon{v_1\. v_2\. v_3\. v_4}$  without self-intersection, where $\theta_1 < \pi$ and $\theta_2 > \pi$}
		\label{fig:NonConQuad}
	\end{center}
\end{figure}
Let the intersection point be denoted by $v_5$. If $v_5$ lies on the edge $\edgepath{v_1\. v_2}$ or $\edgepath{v_3\. v_4}$ as demonstrated in Figure~\ref{fig:NonConQuad}\,(a), it would create a bigon, causing the length of either $\edgepath{v_1\. v_2}$ or $\edgepath{v_3\. v_4}$ to exceed $\pi$. This would violate the assumption that all edge lengths are less than~$\pi$. Therefore, such configurations are not allowed according to our assumptions.

Therefore, the only possible configurations of case (ii) are cases (b) and (c) in Figure~\ref{fig:NonConQuad}. 
For case (b), the quadrilateral $\polygon{v_1 v_2 v_3 v_4}$ encompasses  a hemisphere sharing the geodesic segment $\edgepath{v_2 v_3}$, and we obtain that $\theta_3 > \pi$.
For case (c) on the other hand, 
the interior angles of the triangle $\polygon{v_5\. v_3\. v_4}$ at $v_5$ and $v_4$, as well as the length of $\edgepath{v_5\. v_4}$, are all less than~$\pi$. 
Therefore, by the Sine rule,
we deduce  $\theta_3 <\pi$.
Consequently, cases (b) and (c) cannot occur simultaneously, and it suffices to independently check unique realization in each case.
For case (b), we can decompose the quadrilateral $\polygon{v_1\. v_2\. v_3\. v_4}$ into two pieces: one is a hemisphere and the other is the complementary quadrilateral  $\polygon{v_1\. \widehat{v_2\.v_3}\. v_4}$, where $\widehat{v_2\.v_3}$ denotes the complement segment of $\overline{v_2\.v_3}$ of  the original $\polygon{v_1\. v_2\. v_3\. v_4}$.  
The quadrilateral $\polygon{v_1\. \widehat{v_2\.v_3}\. v_4}$ is 
 is uniquely determined by  the previous case (i), and 
 thus, so is $\polygon{v_1\. v_2\. v_3\. v_4}$.
For case (c), we have $\polygon{v_1\. v_2\. v_3\. v_4} = \polygon{v_1\. v_2\. v_5} \cup \polygon{v_3\. v_4\. v_5}$ and each triangle is uniquely determined by Corollary \ref{trianglerigidity}. 
Specifically, 
the triangle $\polygon{v_1\. v_2\. v_5}$ is uniquely determined by  $ASA$ condition because the length of $\edgepath{v_1\.v_5}$ is less than $\pi$. The triangle $\polygon{v_3\. v_4\. v_5}$ is also uniquely determined by $AAA$ condition because the length of $\edgepath{v_5\.v_4}$ is less than $\pi$.  
	\vspace{0.3em}
	
Let's examine  case (iii) where $\theta_1 < \pi$ and $\theta_2 > \pi$ with self-intersection. 
First, let us consider the case where $\edgepath{v_1\.v_4}$ and $\edgepath{v_2\.v_3}$ intersect, as illustrated in Figure~\ref{fig:intersect4gon}\,(a).
Let the edge $\edgepath{v_3\. v_2}$ extend geodesically from $v_2$ into the interior direction. The extended geodesic must intersect $\edgepath{v_1\.v_2}$ or $\edgepath{v_1\.v_4}$  at $v_5$, as shown in Figure~\ref{fig:intersect4gon}\,(a). 
In either case, the appearance of a bigon $\polygon{v_2\.v_5}$ results in the lengths of $\edgepath{v_1\.v_2}$ or $\edgepath{v_1\.v_4}$ exceeding $\pi$, thereby making case (a) impossible.

	\begin{figure}[h!]
	\begin{center}
		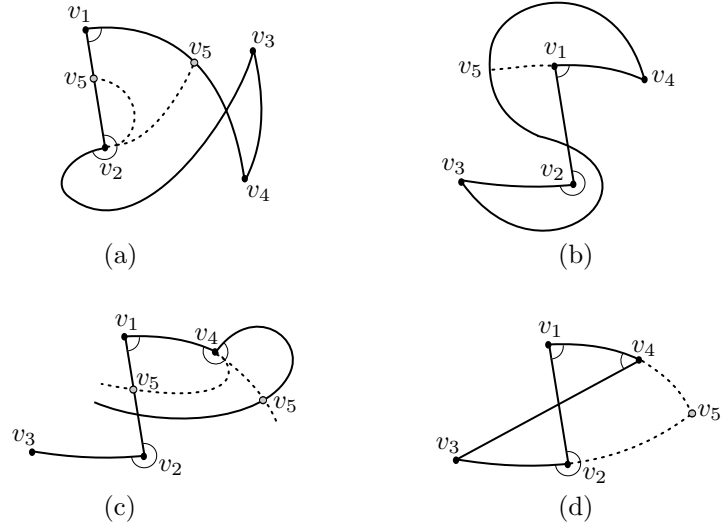
		\caption{ The cases of  self-intersecting  quadrilaterals $(v_1 v_2 v_3 v_4)$ with  $\theta_1  < \pi $ and $\theta_2 > \pi$ }
		\label{fig:intersect4gon}	
	\end{center}
\end{figure}
Recall that adjacent edges cannot intersect and any pair of edges cannot intersect twice because this would involve an edge greater than $\pi$.
As a result,   $\edgepath{v_1\.v_2}$ and $\edgepath{v_3\.v_4}$ must have  exactly one intersection. 
Depending on the direction of the intersection, we  divide this into case (b), and cases (c) or (d) as illustrated in Figure \ref{fig:intersect4gon}. 
Case (b) represents to the situation where $\vect{v_4\.v_3}$ intersects $\vect{v_1\.v_2}$ from its right side, while cases (c) and (d) represents to situations where the intersection occurs from the left side of $\vect{v_1\.v_2}$.
In case (b),  we find the length of $\edgepath{v_3\.v_4}$ to exceed $\pi$, since a bigon $\polygon{v_4\.v_5}$ appears when $\edgepath{v_4\.v_1}$ is extended from $v_1$ in the interior direction. Thus, case (b) is  impossible.
Case (c) corresponds to the situation where the interior angle $\theta_4$ at $v_4$ is greater than $\pi$, while case (d) refers to when $\theta_4$ is less than $\pi$.
In case (c), using a similar approach to the previous cases,	
 we also find the length of $\edgepath{v_1\.v_2}$ or $\edgepath{v_3\.v_4}$  to exceed $\pi$,  since a bigon $\polygon{v_1\.v_5}$ or $\polygon{v_4\.v_5}$ appears when $\edgepath{v_1\.v_4}$ is extended from $v_1$ in the interior direction as shown in Figure~ \ref{fig:intersect4gon}~(c). Thus, case (c) is also impossible.
	
Now, let us consider case (d) as the only possible configuration for case (iii).
We extend the two edges $\edgepath{v_1\. v_4}$ and $\edgepath{v_3\. v_2}$ in the right directions as illustrated in Figure~\ref{fig:intersect4gon}~(d). The two geodesics intersect twice, and we choose the first intersection point as the new vertex $v_5$.
The triangle $\polygon{v_1\. v_2\. v_5}$ is uniquely determined by $ASA$ condition of Proposition~\ref{trianglerigidity} since the length of $\edgepath{v_1\.v_5}$ is less than $\pi$. The triangle $\polygon{v_4\. v_3\. v_5}$ is also uniquely determined by $AAA$ condition since  the length of $\edgepath{v_3\. v_4}$ is less than $\pi$. 
By the uniqueness of the two triangles $\polygon{v_1\. v_2\. v_5}$ and $\polygon{v_4\. v_3\. v_5}$, we obtain the uniqueness of $\polygon{v_1\. v_2\. v_3\. v_4}$ for case~(iii). Note that we deduce that $\theta_4 <\pi$ and  $\theta_3>\pi$  for case (iii).     
	\vspace{0.3em}
	
To complete the proof, we have to check possibility of simultaneous realizations of  case (ii) and case (iii). 
We will demonstrate that $\theta_3$ or $\theta_4$ must differ, thus preventing  simultaneous realization of cases (ii) and (iii).
For case (ii)-(b), we remove the hemisphere and  reconsider the resulting quadrilateral $\polygon{v_1\.\widehat{v_2\. v_3}\. v_4}$. The three interior  angles of $\polygon{v_1\.\widehat{v_2\. v_3}\. v_4}$ at $v_1$, $v_2$, and $v_3$ all are less than~$\pi$. If $\theta_4<\pi$ then all edge lengths of the quadrilateral  are less than~$\pi$ according to Corollary~\ref{lengthSphPolygon}.
	Then, the original edge length $\edgepath{v_2\. v_3}$ becomes greater than $\pi$ and it  contradicts the assumption that all edge lengths  are less than $\pi$. 
	Therefore, it must be  $\theta_4 >\pi$: simultaneous realization of (ii)-(a) and (iii) is impossible.
	In case of (ii)-(c), we have already shown that $\theta_3 < \pi$, making   simultaneous realization with (iii) impossible.
	As a result, we conclude that simultaneous realization of case (ii) and case (iii) is completely impossible. 
	
	Considering all the discussed points, we have ruled out all possibilities for non-unique realizations. This ultimately proves the uniqueness of the quadrilateral $\polygon{v_1\. v_2\. v_3\. v_4}$.
\end{proof} 

We may extend the same theorem to a general $n$-gon with $n>4$.  Providing a rigorous proof is, however, anticipated to be quite challenging. 
 Since it's not critical to our main result of establishing the rigidity of 3-dimensional polyhedra, we have not addressed this issue in the present paper. However, we recognize this question as significant in its own right and leave it as an open problem, as stated in Conjecture~\ref{conj:polygon}.

\bibliographystyle{alpha}
\bibliography{rigidityCK}

\begin{thebibliography}{KRH33}

\bibitem[Ale05]{alexandrov_convex_2005}
A.~D. Alexandrov.
\newblock {\em Convex polyhedra}.
\newblock Springer {Monographs} in {Mathematics}. Springer-Verlag, Berlin,
  2005.

\bibitem[Bel23]{belletti_volume_2023}
Giulio Belletti.
\newblock The volume conjecture for polyhedra implies the {Stoker} conjecture.
\newblock {\em Geometriae Dedicata}, 217(4):77, June 2023.

\bibitem[Con77]{connelly_counterexample_1977}
Robert Connelly.
\newblock A counterexample to the rigidity conjecture for polyhedra.
\newblock {\em Publications Mathématiques de l'IHÉS}, 47:333--338, 1977.

\bibitem[Fis07]{fisher_local_2007}
David Fisher.
\newblock Local rigidity of group actions: past, present, future.
\newblock In {\em Dynamics, ergodic theory, and geometry}, volume~54 of {\em
  Math. {Sci}. {Res}. {Inst}. {Publ}.}, pages 45--97. Cambridge Univ. Press,
  Cambridge, 2007.

\bibitem[Glu75]{gluck_almost_1975}
Herman Gluck.
\newblock Almost all simply connected closed surfaces are rigid.
\newblock In {\em Geometric topology ({Proc}. {Conf}., {Park} {City}, {Utah},
  1974)}, pages 225--239. Lecture Notes in Math., Vol. 438. Springer, Berlin,
  1975.

\bibitem[Gr{\"u}03]{grunbaum_convex_2003}
Branko Gr{\"u}nbaum.
\newblock {\em Convex polytopes}, volume 221 of {\em Graduate {Texts} in
  {Mathematics}}.
\newblock Springer-Verlag, New York, second edition, 2003.
\newblock Prepared and with a preface by Volker Kaibel, Victor Klee and Günter
  M. Ziegler.

\bibitem[IS10]{izmestiev_infinitesimal_2010}
Ivan Izmestiev and Jean-Marc Schlenker.
\newblock Infinitesimal rigidity of polyhedra with vertices in convex position.
\newblock {\em Pacific J. Math.}, 248(1):171--190, 2010.

\bibitem[KRH33]{klein_ubersicht_1933}
Felix Klein, Ernst Ritter, and Otto Haupt.
\newblock Übersicht über die sphärische {Trigonometrie} [*].
\newblock In {\em Vorlesungen über die {Hypergeometrische} {Funktion}:
  {Gehalten} an der {Universität} {Göttingen} im {Wintersemester} 1893/94},
  Die {Grundlehren} der mathematischen {Wissenschaften}, pages 158--196.
  Springer, Berlin, Heidelberg, 1933.

\bibitem[MM11]{mazzeo_infinitesimal_2011}
Rafe Mazzeo and Grégoire Montcouquiol.
\newblock Infinitesimal rigidity of cone-manifolds and the {Stoker} problem for
  hyperbolic and {Euclidean} polyhedra.
\newblock {\em J. Differential Geom.}, 87(3):525--576, 2011.

\bibitem[Pak09]{pak_lectures_2009}
Igor Pak.
\newblock Lectures on discrete and polyhedral geometry.
\newblock {\em Preliminary version available at author's web page}, 2009.

\bibitem[Rat06]{ratcliffe_foundations_2006}
John~G. Ratcliffe.
\newblock {\em Foundations of hyperbolic manifolds}, volume 149 of {\em
  Graduate {Texts} in {Mathematics}}.
\newblock Springer, New York, second edition, 2006.

\bibitem[Rot81]{roth_rigid_1981}
B.~Roth.
\newblock Rigid and {Flexible} {Frameworks}.
\newblock {\em The American Mathematical Monthly}, 88(1):6--21, 1981.

\bibitem[Sch00]{schlenker_dihedral_2000}
J.-M. Schlenker.
\newblock Dihedral angles of convex polyhedra.
\newblock {\em Discrete \& Computational Geometry. An International Journal of
  Mathematics and Computer Science}, 23(3):409--417, 2000.

\bibitem[Sto68]{stoker_geometrical_1968}
J.~J. Stoker.
\newblock Geometrical problems concerning polyhedra in the large.
\newblock {\em Communications on Pure and Applied Mathematics}, 21:119--168,
  1968.

\end{thebibliography}

\end{document}